\documentclass[brochure,12pt]{bourbaki}
\usepackage{amssymb,amsfonts,amsmath,footnote}
\usepackage{upmath,mathrsfs,xspace}
\usepackage{url}

\usepackage{smfhref}

\usepackage[utf8]{inputenc}
\DeclareUnicodeCharacter{00A0}{\relax}
\usepackage[T1]{fontenc}
\usepackage[frenchb]{babel}

\addressindent \textwidth
\setbox0\hbox{Université de Rennes~1 \&\ Institut universitaire de France}
\advance\addressindent-\wd0

\date{Janvier 2011}
\bbkannee{63\textsuperscript e ann\'ee, 2010--2011}
\bbknumero{1032}
\title{Relations de dépendance \\ et intersections exceptionnelles}
\author{Antoine CHAMBERT-LOIR}
\address{Universit\'e de Rennes~1 \&\ Institut universitaire de France\\
Irmar \&\ UFR de mathématiques \\
Campus de Beaulieu\\
F--35042 Rennes Cedex}
\email{antoine.chambert-loir@univ-rennes1.fr}

\def\cf{\emph{cf.}\nobreak\xspace}
\def\resp{\emph{resp.}\nobreak\xspace}
\def\codim{\operatorname{codim}}
\def\Id{\operatorname{Id}}
\def\rg{\operatorname{rang}}
\def\End{\operatorname{End}}
\def\Pic{\operatorname{Pic}}
\def\NS{\operatorname{NS}}
\def\Gal{\operatorname{Gal}}
\def\GL{\operatorname{GL}}
\def\Z{\mathbf Z}
\def\N{\mathbf N}
\def\C{\mathbf C}
\def\F{\mathbf F}
\def\Q{\mathbf Q}
\def\P{\mathbf P}
\def\R{\mathbf R}
\def\bQ{\overline{\mathbf Q}}
\def\gm{\mathbf G_{\mathrm m}}
\def\tors{{\mathrm {tors}}}
\def\oa{{\mathrm {oa}}}
\def\ta{{\mathrm {ta}}}
\def\ano{{\mathrm {ano}}}
\let\ra\rightarrow
\let\bar\overline
\let\phi\varphi\let\eps\varepsilon
\let\leq\leqslant \let\geq\geqslant
\def\abs#1{\left|#1\right|}
\def\norm#1{\left\|#1\right\|}
\begin{document}

\begin{abstract}
L'exposé sera consacré au résultat suivant, issus des travaux de
Bombieri, Masser, Zannier et Maurin: Soit $C$ une courbe algébrique
(irréductible) complexe  et considérons $n$ fonctions
rationnelles $f_1,\dots,f_n$ non identiquement nulles
et multiplicativement indépendantes sur $C$. 
Les points~$x$ de~$C$ où leurs valeurs $f_1(x),\dots,f_n(x)$
vérifient au moins deux relations de dépendance multiplicative
indépendantes forment un ensemble fini.

Nous discuterons les généralisations conjecturales de ce théorème
(Bombieri, Masser, Zannier; Zilber; Pink) concernant la finitude
des points d'une sous-variété~$X$ de dimension~$d$ d'une variété
semi-abélienne~$A$ qui appartiennent à un sous-groupe algébrique
de codimension~$>d$ dans~$A$, leurs relations avec les théorèmes
de type Mordell--Lang ou Manin--Mumford et, dans le cas arithmétique,
les résultats récents (Habegger; Rémond) concernant la hauteur des
points appartenant à un sous-groupe algébrique de codimension~$d$.
\end{abstract}

\maketitle


\section{Relations de dépendance}\label{sec.depend}

Cet exposé est consacré à un ensemble de travaux 
apparus depuis une quinzaine d'années sous la plume de divers mathématiciens
autour de ce qu'on appelle maintenant la \emph{conjecture de Zilber--Pink}.
Je voudrais débuter avec le cas le plus simple et le plus frappant.

\begin{theo}\label{theo.maurin}
Soit $C$ une courbe algébrique complexe (irréductible)
et considérons $n$ fonctions rationnelles $f_1,\dots,f_n$ 
non identiquement nulles et multiplicativement indépendantes sur~$C$.
Alors, les points~$x$ de~$C$  où leurs valeurs $f_1(x),\dots,f_n(x)$
vérifient au moins deux relations de dépendance multiplicative
indépendantes forment un ensemble fini.
\end{theo}

Dire que $f_1,\dots,f_n$ sont multiplicativement indépendantes
signifie que pour tout vecteur non nul $(a_1,\dots,a_n)\in\Z^n$,
la fonction rationnelle $f_1^{a_1}\dots f_n^{a_n}$ sur~$C$
n'est pas constante de valeur~$1$.

De même, si $x$ est un point de~$C$ qui n'est ni un zéro ni un pôle
des~$f_i$, les relations de dépendance multiplicative entre les valeurs
$f_i(x)$ des~$f_i$ sont les vecteurs $(a_1,\dots,a_n)$ de~$\Z^n$
tels que $\prod_{i=1}^n f_i(x)^{a_i}=1$. Ces relations forment un sous-groupe
de~$\Z^n$; le théorème concerne les points~$x$ pour lesquels
ce sous-groupe est de rang~$\geq 2$.

Sous l'hypothèse que les fonctions~$f_i$ sont multiplicativement
indépendantes modulo les constantes, c'est-à-dire qu'aucune
combinaison non triviale $\prod f_i^{a_i}$ n'est constante,
ce théorème a été démontré par E.~\textsc{Bombieri}, D.~\textsc{Masser}
et U.~\textsc{Zannier} dans l'article~\cite{bombieri-masser-zannier:1999}
qui, le premier, a mis en avant  ces questions.
L'hypothèse supplémentaire a été levée par
G.~\textsc{Maurin}~\cite{maurin:2008}, puis, par une autre approche,
par \textsc{Bombieri}, P.~\textsc{Habegger}, \textsc{Masser} 
et \textsc{Zannier}~\cite{bombieri-habegger-masser-zannier:2010}.
Ces articles reposent sur des techniques de géométrie
diophantienne et supposent en outre que toute la situation
est définie sur le corps~$\bQ$ des nombres algébriques.
Dans l'intervalle, l'article~\cite{bombieri-masser-zannier:2008b}
démontre que ce cas entraîne le théorème sur le corps des nombres
complexes.

Remarquons pour finir que l'énoncé est optimal au sens
où l'ensemble des points~$x$ de~$C$ où les valeurs $f_1(x),\dots,f_n(x)$
sont multiplicativement dépendantes est \emph{infini}
(à moins que  les~$f_i$ ne soient toutes constantes).
Si, disons, $f_1$ n'est pas constante,
il suffit de considérer l'ensemble des points~$x$ de~$C$
tels que $f_1(x)$ soit une racine de l'unité.

\medskip

Malgré la simplicité de l'énoncé ci-dessus, il convient de l'écrire
dans le cadre plus général des groupes algébriques commutatifs,
voire des variétés de Shimura mixtes.
Il s'insère alors dans un faisceau de conjectures dues
à B.~\textsc{Zilber}~\cite{zilber:2002},
S.-W.~\textsc{Zhang} (non publié, voir~\cite{bombieri-masser-zannier:2006}),
R.~\textsc{Pink}~\cite{pink:2005b} (voir aussi~\cite{pink:2005})
et
\textsc{Bombieri}, \textsc{Masser}, \textsc{Zannier}~\cite{bombieri-masser-zannier:2006}.
Par ailleurs, et ce n'est pas l'aspect le moins fascinant du sujet,
ces conjectures complètent les conjectures de type 
Manin--Mumford, Mordell--Lang et André--Oort.
Toutefois, je me limiterai
au cas des groupes algébriques dans ce rapport et
ne dirai rien de la conjecture d'André--Oort
dont l'importance et la beauté des derniers développements, 
dus pour l'essentiel à J.~\textsc{Pila} (voir~\cite{pila:2011}),
exigent qu'un exposé autonome leur soit consacré.
Signalons quand même qu'ils trouvent leur origine
dans la nouvelle démonstration 
de la conjecture de Manin--Mumford
qu'ont découverte \textsc{Pila} et \textsc{Zannier}~\cite{pila-zannier:2008}
et que ces techniques jouent un rôle dans l'étude de questions
voisines que nous évoquerons à la fin de ce rapport.

Revenons au théorème~\ref{theo.maurin}.
Puisque l'on discute de relations de dépendance multiplicative,
introduisons donc le groupe multiplicatif (complexe)
$\gm=\mathbf A^1\setminus\{0\}$ et considérons la
famille des fonctions $f_1,\dots,f_n$
comme une application rationnelle~$f$ de~$C$ 
dans le \emph{tore algébrique}~$G=\gm^n$.
Notons $X$ son image, ou plutôt l'adhérence,
pour la topologie de Zariski dans~$G$, de l'image
d'un ouvert dense de~$C$ sur lequel $f$ est définie.
Laissons de côté le cas inintéressant où les~$f_i$ sont toutes constantes;
l'application~$f$ est alors de degré fini 
et $X$ est une courbe irréductible dans~$G$.
Comme les~$f_i$ sont multiplicativement indépendantes, 
la courbe~$X$ n'est contenue dans aucun sous-groupe algébrique strict
de~$G$: de tels sous-groupes sont en effet définis par des équations monomiales
$g_1^{a_1}\dots g_n^{a_n}=1$ en les coordonnées $(g_1,\dots,g_n)\in\gm^n$.
Plus précisément, les sous-groupes algébriques (pas forcément connexes)
de~$G$ sont en bijection avec les sous-modules de~$\Z^n$,
la codimension d'un sous-groupe étant égale au rang du module de ses relations.
Pour tout entier~$r\in\N$,
notons ainsi $G^{[r]}$ la réunion des sous-groupes algébriques de~$G$
qui sont de codimension~$\geq r$; c'est aussi l'ensemble des points
$(g_1,\dots,g_n)\in\gm^n$ qui vérifient $r$~relations de dépendance
multiplicative indépendantes.
Ainsi, le théorème~\ref{theo.maurin}
équivaut à l'énoncé suivant:
{\setcounter{defi}{0}\def\thedefi{\thesection.\arabic{defi}$'$}
\begin{theo}\label{theo.maurin'}
Soit $X$ une sous-variété fermée de~$\gm^n$,  irréductible et de dimension~$1$,
qui n'est contenue dans aucun sous-groupe algébrique strict.
L'ensemble $X\cap G^{[2]}$ 
des points~$x$ de~$X$ qui sont contenus dans un sous-groupe
de codimension~$\geq 2$ est fini.
\end{theo}}

\medskip

Les conjectures auxquelles ce rapport est consacré
visent à remplacer le groupe~$G$ par une variété semi-abélienne,
la courbe~$X$ par une sous-variété fermée de~$G$,
irréductible et distincte de~$G$, 
de dimension quelconque~$d$,
et l'ensemble~$G^{[2]}$ par un ensemble~$G^{[c]}$, où $c$ est un entier
tel que $c>d$,
voire un ensemble de la forme $\Gamma\cdot G^{[c]}$
où $\Gamma$ est un sous-groupe de rang fini de~$G$,
et même, lorsque tout est défini sur le corps des nombres
algébriques, des \og  épaississements\fg d'un tel ensemble au  
sens de la théorie des hauteurs.

Dans ce cas, une généralisation naturelle du théorème~\ref{theo.maurin'} 
est la conjecture suivante de R.~\textsc{Pink}
(voir~\cite{pink:2005}, conjecture~5.1).
Avant de l'énoncer, 
rappelons qu'une variété semi-abélienne est un groupe algébrique commutatif
qui se décompose en une extension d'une variété abélienne par un tore;
un sous-groupe algébrique connexe d'une variété  semi-abélienne
est encore une variété semi-abélienne.
\begin{conj}\label{conj.pink}
Soit $G$ une variété semi-abélienne complexe.
Soit $X$ une sous-variété fermée et irréductible de~$G$, 
de dimension~$d$.
Si $X$ n'est pas contenue dans un sous-groupe algébrique
strict de~$G$, l'intersection $X\cap G^{[d+1]}$ 
n'est pas dense dans~$X$ pour la topologie de Zariski.
\end{conj}

Cette conjecture est connue lorsque $X$ est une courbe ($d=1$) et $G$ un tore;
on retrouve le théorème~\ref{theo.maurin'}.
Pour l'essentiel, tous les autres résultats se restreignent
au cas de variétés définies sur le corps~$\bQ$ des nombres algébriques 
et peuvent donc ne considérer que des points algébriques. 
Toujours lorsque $d=1$,
un théorème de G.~\textsc{Rémond} 
établit cette variante pour les variétés abéliennes à multiplications complexes
(\cite{remond:2007}, corollaire~1.6), tandis qu'un résultat plus récent
de E.~\textsc{Viada} (\cf\cite{galateau:2010}, théorème~H)
traite le cas des variétés abéliennes
qui sont produit de variétés abéliennes à multiplications complexes, 
surfaces abéliennes et courbes elliptiques.
En dimension plus grande, à l'exception de quelques cas 
comme les sous-variétés
de codimension~$2$ d'un tore 
(théorème~1.7 de~\cite{bombieri-masser-zannier:2007}),
cette conjecture n'est démontrée que sous une hypothèse 
géométrique sur~$X$.
Introduisons une terminologie proposée par Z.~\textsc{Ran}
dans le cas des variétés abéliennes (voir~\cite{ran:1981}):
\begin{defi}
Soit $X$ une sous-variété (fermée, irréductible)
d'une variété semi-abélienne~$G$.
On dit que $X$ est \emph{géométriquement dégénérée} 
s'il existe un sous-groupe algébrique~$G'$ de~$G$, 
tel que l'image de~$X$ dans~$G/G'$
est de dimension strictement inférieure à~$\min(\dim (X),\dim (G/G'))$.
\end{defi}
Pour qu'une courbe soit géométriquement dégénérée,
il faut et il suffit qu'elle soit contenue
dans un translaté d'un sous-groupe algébrique strict.
En revanche, en dimension supérieure, il peut exister des sous-variétés
géométriquement dégénérées
qui ne sont contenues dans aucun translaté de sous-groupe algébrique.

\begin{theo}[\textsc{Habegger}, \cite{habegger:2009a,habegger:2009d}]
\label{theo.nd}
Soit $G$ une variété semi-abélienne  définie sur~$\bQ$.
Soit $X$ une sous-variété fermée, irréductible et de dimension~$d$,
définie sur~$\bQ$,
qui n'est pas géométriquement dégénérée.
Supposons de plus que $G$ soit un tore 
ou une variété abélienne à multiplications complexes.
Alors $X(\bQ)\cap G^{[d+1]}$ n'est pas dense dans~$X$
pour la topologie de Zariski.
\end{theo}

\medskip

Expliquons maintenant le lien entre la conjecture~\ref{conj.pink}
et celles de Manin--Mumford ou de Mordell--Lang.

Un point~$g$ de~$G$ est de torsion si et seulement s'il appartient
à un sous-groupe algébrique de dimension~$0$ (celui qu'il engendre!);
par suite, la conjecture~\ref{conj.pink} étend la conjecture de Manin--Mumford
qui concerne précisément les points de torsion de~$G$ 
situés sur la sous-variété~$X$.
Rappelons que cette conjecture a été démontrée 
par M.~\textsc{Laurent}~\cite{laurent84} pour les tores,
M.~\textsc{Raynaud}~\cite{raynaud83,raynaud83b} dans le cas des variétés
abéliennes, et M.~\textsc{Hindry}~\cite{hindry1988} en général.
Profitons aussi de l'occasion pour mentionner diverses preuves plus récentes:
\cite{sarnak-adams:1994},
via la preuve d'une conjecture de Bogomolov
(\cite{zhang95,bombieri-zannier:1995,david-p99,zannier:2009}  pour les tores,
\cite{ullmo98,zhang98,david-p98} pour les variétés abéliennes,
\cite{david-p2000} pour les variétés semi-abéliennes),
\emph{via} la théorie des modèles de corps aux différences
(\cite{hrushovski2001}
et \cite{pink-roessler:2002,pink-roessler:2004,roessler:2005}
qui s'en inspirent),
enfin \cite{pila-zannier:2008}.

Soit maintenant $\Gamma$ un sous-groupe de~$G$ de rang fini~$r$;
soit $(g_1,\dots,g_r)$ des éléments de~$\Gamma$
tels que $\Gamma/\langle g_1,\dots,g_r\rangle$ soit de torsion.
Reprenant un argument de \textsc{Rémond} qui consiste
à remplacer~$G$ par une puissance~$G\times G^{r}$
et $X$ par la sous-variété  $X\times\{(g_1,\dots,g_r)\}$,
\textsc{Pink} observe que la conjecture~\ref{conj.pink} est équivalente
à la conjecture suivante:
\begin{conj}\label{conj.pink-ml}
Soit $G$ une variété semi-abélienne complexe.
Soit $X$ une sous-variété fermée de~$G$, irréductible et distincte de~$G$,
de dimension~$d$.
Soit $\Gamma$ un sous-groupe de rang fini de~$G$.
Si $X$ n'est pas contenue dans un translaté de sous-groupe algébrique
strict de~$G$, l'intersection $X\cap (\Gamma\cdot G^{[d+1]})$ 
n'est pas dense dans~$X$ pour la topologie de Zariski.
\end{conj}
Bien sûr, si $X\neq G$, $G^{[d+1]}$ n'est pas vide
et $\Gamma\cdot G^{[d+1]}$ contient~$\Gamma$,
donc la conjecture~\ref{conj.pink-ml}
entraîne en particulier 
la conjecture de Mordell--Lang selon laquelle
$X\cap\Gamma$ n'est pas dense dans~$X$.
Rappelons que celle-ci est maintenant un théorème, suite aux
travaux de P.~\textsc{Liardet}~\cite{liardet:1975}
pour les courbes dans les tores, M.~\textsc{Laurent}~\cite{laurent84}
pour les tores,
  G.~\textsc{Faltings}~\cite{faltings1991} pour les variétés abéliennes,
M.~\textsc{Hindry}~\cite{hindry1988},
  P.~\textsc{Vojta}~\cite{vojta1996}
et M.~\textsc{McQuillan}~\cite{mcquillan1995} 
pour les variétés semi-abéliennes.

\textsc{Rémond} et \textsc{Maurin} ont tiré parti
de la méthode de \textsc{Vojta} pour établir
le résultat suivant en direction de la conjecture~\ref{conj.pink-ml},
analogue avec un groupe~$\Gamma$ du théorème~\ref{theo.nd}.

\begin{theo}[\textsc{Rémond}~\cite{remond:2009},
\textsc{Maurin}~\cite{maurin:2011}]
\label{theo.nd.gamma}
Supposons que $G$ soit un tore 
ou une variété abélienne à multiplications complexes.
Soit $X$ une sous-variété fermée et irréductible de dimension~$d$
de~$G$, 
définie sur~$\bQ$ et qui n'est pas géométriquement dégénérée.
Soit $\Gamma$ un sous-groupe de rang fini de~$G(\bQ)$.
Alors $X(\bQ)\cap (\Gamma\cdot G^{[d+1]})$ n'est pas dense dans~$X$
pour la topologie de Zariski.
\end{theo}

Leur approche permet aussi d'obtenir une autre démonstration
du théorème~\ref{theo.nd}.

Par ailleurs, une approche récente d'E.~\textsc{Viada}~\cite{viada:2010}
fondée sur une forme effective de la conjecture de Bogomolov
permet d'étendre le théorème précédent
à une classe de variétés abéliennes comprenant
les produits de variétés abéliennes à multiplications complexes,
de courbes elliptiques et de surfaces abéliennes.


\medskip

Dans le \S2 de ce rapport, je décris les {\og lieux exceptionnels\fg}
que l'étude de ces conjectures oblige à prendre en compte.
Le \S3 est un rappel des concepts de géométrie diophantienne
utilisés dans les preuves: théorie des hauteurs, problème de Lehmer,
conjecture de Bogomolov. 
Dans le \S4, je résume les différentes approches des théorèmes
précédents: après la preuve d'un important
théorème de majoration de hauteurs, j'y décris
successivement l'utilisation du problème de Lehmer, de la méthode de Vojta 
et des formes effectives de la conjecture de Bogomolov.

\section{Intersections exceptionnelles}

Motivé par l'étude de la conjecture de Schanuel
dans le cadre des corps différentiels,
B.~\textsc{Zilber} avait énoncé une conjecture voisine
de la conjecture~\ref{conj.pink}
(\cite{zilber:2002}, conjecture~2).

Soit $G$ une variété semi-abélienne complexe.
Soit $X$ une sous-variété fermée de~$G$, irréductible.

\begin{defi}\label{defi.atypique}
Soit $H$ une sous-variété fermée,  équidimensionnelle de~$G$.
On dit qu'une composante irréductible~$Y$ de l'intersection $X\cap H$
est \emph{atypique}
si sa dimension vérifie
\[ \dim (Y) > \dim (X) + \dim (H)-\dim (G). \]
\end{defi}
On peut aussi écrire cette inégalité sous la forme $\codim_H(Y)<\codim_G(X)$.
Rappelons que si $X\cap H$ n'est pas vide, la théorie de l'intersection affirme
que toute composante irréductible en est de dimension supérieure
ou égale à $\dim (X)+\dim (H)-\dim (G)$;
les composantes atypiques sont donc celles dont la dimension dépasse
la dimension attendue.

Dans la suite, cette notion n'interviendra que dans le cas particulier
où $H$ est un sous-groupe algébrique de~$G$,
ou un translaté d'un tel sous-groupe.

\begin{conj}\label{conj.zilber}
Soit $G$ une variété semi-abélienne complexe.
Soit $X$ une sous-variété fermée de~$G$, irréductible.
Il existe une famille \emph{finie} $\Phi$ de sous-groupes
algébriques stricts de~$G$ telle que pour tout sous-groupe algébrique~$H$
de~$G$, toute composante atypique de~$X\cap H$ soit contenue
dans l'un des éléments de~$\Phi$.
\end{conj}

Avec ces notations,
notons $X^\ta$ le complémentaire dans~$X$
de la réunion des composantes atypiques 
de dimension strictement positive d'une intersection~$X\cap H$,
où $H$ est un sous-groupe algébrique de~$G$.\footnote
{Les lettres ta signifient \emph{torsion anomalous}.}
Une variante de la conjecture~\ref{conj.zilber} due
à \textsc{Bombieri}, \textsc{Masser} et \textsc{Zannier}
(\cite{bombieri-masser-zannier:2007}, \emph{Torsion Openness Conjecture},
p.~25) affirme que $X^\ta$ est ouvert dans~$X$.

\textsc{Zilber} observe aussi dans son article (\cite{zilber:2002},
propositions~2 et~3)
comment la conjecture~\ref{conj.zilber}
entraîne celles de Manin--Mumford ou de Mordell--Lang.

C'est d'ailleurs très simple dans le cas Manin--Mumford. 
Considérons en effet une sous-variété (fermée, irréductible)~$X$
de~$G$, distincte de~$G$ et qui n'est pas contenue dans un sous-groupe
algébrique strict de~$G$.
Soit alors~$x$ un point de~$X$
qui est de torsion dans~$G$ et soit $H_x$ le sous-groupe algébrique
qu'il engendre. Alors, $Y=\{x\}$ est une composante irréductible
de l'intersection $X\cap H_x$ ; puisque $X\subsetneq G$,
la dimension de~$Y$ vérifie
\[ \dim (Y)=0 >\dim (X)-\dim (G) =\dim (X)+\dim ( H_x)-\dim (G),\] 
si bien que $Y$ est une composante atypique de cette intersection.
Supposant la conjecture~\ref{conj.zilber} satisfaite, 
il existe $H\in\Phi$ tel que $x\in H$.
Autrement dit, les points de~$X$ qui sont de torsion dans~$G$ sont contenus
dans la réunion finie des sous-variétés $X\cap H$, pour $H\in\Phi$.
À moins que $X\cap H=X$ pour l'un de ces sous-groupes, 
ce qui entraîne que $X$ est
contenue dans un sous-groupe algébrique strict,
les points de~$X$ qui sont de torsion dans~$G$
ne sont donc pas denses dans~$X$ pour la topologie de Zariski.

\medskip

Le théorème suivant est dû à \textsc{Zilber}~\cite{zilber:2002}
lorsque $G$ est un tore et à J.~\textsc{Kirby}~\cite{kirby:2009}
en général. La différence avec la conjecture~\ref{conj.zilber}
vient du fait qu'on met en jeu tous les translatés
de sous-variétés semi-abéliennes et pas seulement
les sous-groupes algébriques (qui sont réunion finie
de translatés de sous-variétés semi-abéliennes
par des \emph{points de torsion}).

\begin{theo}\label{theo.kirby}
Soit $G$ une variété semi-abélienne complexe
et soit $X$ une sous-variété (fermée, irréductible) de~$G$.
Il existe une famille \emph{finie}~$\Phi$ de sous-variétés
semi-abéliennes strictes de~$G$ telle que:
pour tout translaté $gK$ d'une sous-variété semi-abélienne~$K$ de~$G$
et toute composante atypique~$Y$ de l'intersection $X\cap gK$,
il existe $H\in\Phi$ et $h\in G$ tels que $Y\subset h H$
et
\[ \dim (H)+\dim(Y)=\dim(K)+\dim(X\cap hH). \]
\end{theo}

Cette dernière condition de dimension
entraîne que la codimension de~$H$ est au moins
égale à l'excès~$t$ de dimension
de la composante atypique~$Y$, excès que l'on  définit par
\[ t=\dim (Y)-(\dim (X)+\dim (K)-\dim (G)). \]
Au moins lorsque $H$ contient~$K$,
elle signifie que $Y$
est une composante \emph{typique} de l'intersection
de~$X\cap hH$ avec~$gK$ dans le translaté~$gH$.

Le théorème~\ref{theo.kirby}
et son précurseur sur les tores (\cite{zilber:2002}, corollaire~3;
voir aussi~\cite{poizat:2001})
reposent sur un théorème de J.~\textsc{Ax} établissant une
variante de la conjecture de Schanuel dans le cas différentiel
(\cite{ax:1972}, théorème~3).
Par des arguments de théorie des modèles 
des corps différentiels (le théorème de compacité
en logique du premier ordre), \textsc{Zilber}  et \textsc{Kirby}
en déduisent une version uniforme à paramètres,
puis l'énoncé ci-dessus
avec la condition $\codim(H)\geq t$.
Lorsque $Y$ est une composante atypique
de l'intersection $(X\cap hH)\cap gK$, on applique de nouveau l'argument.

Signalons que \textsc{Bombieri}, \textsc{Masser} et \textsc{Zannier}
dans le cas des tores
(\cite{bombieri-masser-zannier:2007}, théorème~1.4),
puis \textsc{Rémond} dans le cas des variétés abéliennes
(\cite{remond:2009}, partie~3),
démontrent des résultats de même nature, mais \emph{effectifs},
au sens où ils contrôlent le degré d'une partie maximale
qui est une composante atypique d'une intersection de~$X$
avec un translaté de sous-groupe algébrique.

\begin{coro}\label{coro.oa}
Soit $X^\oa$ le complémentaire dans~$X$ des composantes
atypiques de dimension strictement positive
d'une intersection $X\cap gH$, où $H$ parcourt l'ensemble
des sous-variétés semi-abéliennes non nulles
de~$G$ de codimension~$\leq\dim (X)$ et $g$ parcourt~$G$.
Alors $X^\oa$ est ouvert dans~$X$ pour la topologie de Zariski.\footnote
{Les lettres oa sont l'abréviation de \emph{open anomalous}.}
\end{coro}
\begin{proof}
On commence par remarquer que si $H$ est une sous-variété semi-abélienne
fixée, la réunion~$\mathscr Z_H$ des composantes atypiques de dimension
strictement positive d'une intersection $X\cap hH$ 
est une partie fermée de~$X$. Considérons en effet
le morphisme $\phi\colon G\ra G/H$ et sa restriction
$\phi_X\colon X\ra\phi(X)$ à~$X$ ; on 
a $X\cap xH=\phi_X^{-1}(\phi_X(x))$ pour tout $x\in X$.
Les composantes en question
sont celles des fibres de~$\phi_X$ qui sont de dimension~$>\max(0,\dim(X)-\dim(G/H))$.
D'après le théorème de semi-continuité
de la dimension des fibres d'un morphisme, c'est une partie fermée de~$X$.

Nous allons démontrer que $X\setminus X^\oa$ est la réunion
des parties~$\mathscr Z_H$, où $H$ décrit l'ensemble fini~$\Phi$
du théorème~\ref{theo.kirby}.

Considérons donc une sous-variété~$Y$ de~$X$ 
de dimension strictement positive, composante irréductible
atypique d'une intersection $X\cap yK$, où $y\in Y$
et $K$ est une sous-variété semi-abélienne de~$G$.
On peut supposer,
quitte à remplacer~$K$ par cette sous-variété,
que $K$ est la plus petite sous-variété semi-abélienne
de~$G$ telle que $Y\subset yK$. Sous l'hypothèse
supplémentaire que $Y$ est maximale parmi les composantes
atypiques, nous allons démontrer que $K\in\Phi$;
nous aurons donc $Y\subset\mathscr Z_K$,
puis $Y\subset \bigcup_{H\in\Phi} \mathscr Z_H$
et le corollaire en résultera d'après la première partie
de la démonstration.

Soit $H\in\Phi$ la sous-variété semi-abélienne de~$G$ fournie par
le théorème~\ref{theo.kirby}: on a $Y\subset yH$
et \[ \dim(H)+\dim(Y)=\dim(K)+\dim(X\cap yH). \]
Par définition de~$K$, on a $K\subset H$.
Soit $Y'$ une composante irréductible de l'intersection~$X\cap yH$
qui contient~$Y$.  L'inégalité 
$\dim(Y)\geq\dim(Y')-\codim_H(K)$  donnée par la théorie
de l'intersection entraîne que $\dim(Y')=\dim(X\cap yH)$.
Comme $Y$ est une sous-variété irréductible maximale
dans l'ensemble des composantes atypiques d'intersections
de~$X$ avec un translaté d'une sous-variété semi-abélienne,
on a l'alternative suivante:
\begin{itemize}
\item Si $Y'=Y$, l'égalité de dimensions ci-dessus entraîne $H=K$;
\item Sinon, $Y\subsetneq Y'$, donc $Y$ est une composante
typique de l'intersection $Y'\cap yK$, c'est-à-dire
\[ \dim (Y') = \dim(X)+\dim(H)-\dim(G)
= \dim(Y)+\dim(H)-\dim(K) ;\]
on en déduit que
\[ \dim(Y)=\dim(X)+\dim(K)-\dim(G), \]
ce qui contredit l'hypothèse que $Y$ est une composante atypique
de l'intersection $X\cap yK$.
\end{itemize}
Cela conclut la preuve du corollaire.
\end{proof}

Notons que $X^\oa$ peut fort bien être vide; c'est
le cas si 
$X$ est contenu dans un translaté de sous-variété semi-abélienne,
plus généralement si $X$ est géométriquement dégénérée
(cf. \cite{remond:2009}, Prop. 4.2):
\begin{prop}
L'ouvert $X^\oa$ est vide si et seulement si $X$ est géométriquement
dégénérée.

Plus généralement, $X\setminus X^\oa$ est la réunion des
sous-variétés~$Y$ de~$X$ pour lesquelles il existe une sous-variété
semi-abélienne~$H$ telle que 
\[ \dim(YH/H)<\min(\dim(Y),\dim(G/H)-\codim_X(Y)). \]
\end{prop}
\begin{proof}
Soit $Y$ une sous-variété de~$X$ et $H$ une sous-variété semi-abélienne
de~$G$ telle que $\dim(YH/H)<\min(\dim(Y),\dim(G/H)-\codim_X(Y))$.
Notons $\phi$ la projection de~$G$ sur~$G/H$ 
et $\phi_Y\colon Y\ra \phi(Y)=YH/H$ le morphisme qui s'en
déduit par restriction à~$Y$.
Pour tout $y\in Y$, on a
\begin{align*}
\dim_y(X\cap yH) & \geq \dim_y(Y\cap yH) = \dim_y \phi_Y^{-1}(\phi_Y(y)) \\
& \geq \dim(Y)-\dim(\phi_Y(Y)) \\
& > \dim(Y)-\min(\dim(Y),\dim(G/H)-\dim(X)+\dim(Y)) \\
& =\max(0,\dim(X)-\codim_G(H)).
\end{align*}
Cela prouve qu'au moins une composante irréductible 
de l'intersection $X\cap yH$ 
passant par~$y$
est atypique. Par suite, $y\not\in X^\oa$.

Inversement, il suffit de démontrer que toute composante
irréductible~$Y$ d'un ensemble~$\mathscr Z_H$ introduit
dans la preuve du corollaire~\ref{coro.oa}
vérifie cette inégalité. Or, par la définition même de~$\mathscr Z_H$,
toute fibre du morphisme~$\phi_Y\colon Y\ra YH/H\subset G/H$
est de dimension~$>\max(0,\dim(X)+\dim(H)-\dim(G))$.
Par conséquent,
\begin{align*}
 \dim(YH/H) & < \dim(Y) - \max(0,\dim(X)+\dim(H)-\dim(G)) \\
&= \min(\dim(Y),\dim(G/H)-\codim_X(Y)), \end{align*}
comme annoncé.
\end{proof}

Plus généralement, si $t$ est un entier naturel,
on peut définir l'ensemble~$X^{\oa,[t]}$
comme le complémentaire, dans~$X$, de l'ensemble des points~$x\in X$
tels qu'il existe une sous-variété semi-abélienne~$H$ de~$G$ telle que
\[ \dim_x (X\cap xH) \geq \max(1,t-\codim_G(H)). \]
Lorsque $t=1+\dim (X)$, on a $X^{\oa,[t]}=X^\oa$. 
Inversement, lorsque $t\leq \dim (X)$
et \mbox{$\dim (X)>0$,} on a $\dim_x(X\cap xG) =\dim(X)>\max(0,t)$ pour tout
$x\in X$, ce qui entraîne que $X^{\oa,[t]}=\emptyset$.

La proposition suivante généralise le corollaire~\ref{coro.oa}.
\begin{prop}
Pour tout entier naturel~$t$, $X^{\oa, [t]}$ est ouvert dans~$X$.
\end{prop}

\section{Hauteurs}

Le reste de ce rapport est consacré à résumer
comment procèdent les preuves des théorèmes~\ref{theo.maurin'}
ou~\ref{theo.nd}.

Tout d'abord, il s'agit de démonstrations où l'arithmétique
joue un rôle fondamental, par la considération 
des hauteurs, et surtout par l'utilisation de  minorations
extrêmement fines des hauteurs dans le contexte
du problème de Lehmer et de ses variantes,
ou du problème de Bogomolov.
De sorte, même si j'avais donné jusqu'à présent
des énoncés sur le corps des nombres complexes, 
tous les arguments qui suivent supposent que les variétés considérées
sont définies sur le corps~$\bQ$ des nombres algébriques.

Comme je l'ai dit à propos du théorème~\ref{theo.maurin},
\textsc{Bombieri}, \textsc{Masser} et \textsc{Zannier}
ont expliqué dans leur article~\cite{bombieri-masser-zannier:2008b}
comment l'on peut démontrer des résultats sur~$\C$
lorsqu'on dispose de la conjecture~\ref{conj.pink} sur~$\bQ$,
voire des cas particuliers de la conjecture 
obtenus en bornant $\dim (X)$ et $\dim (G)$. Ils se sont limités au cas
des tores et il serait intéressant de développer ces arguments
dans le cas général
des variétés semi-abéliennes et même dans le contexte que considère
\textsc{Pink} dans~\cite{pink:2005}.
(Dans le cadre de la conjecture de Mordell--Lang sur les
variétés semi-abéliennes,
des arguments de spécialisation se trouvent dans~\cite{hindry1988}.)

En outre, la littérature se divise entre 
tores et variétés abéliennes, le cas général des variétés semi-abéliennes
n'étant, à ma connaissance, pas encore traité.

\subsection{La machine des hauteurs}

Rappelons rapidement la notion de hauteur sur l'ensemble
des points algébriques d'une variété projective.

Sur l'espace projectif, la hauteur standard
est la fonction $h\colon\P^n(\bQ)\ra\R_+$
définie par la formule
\[ h(x) = \frac1{[K:\Q]} \sum_v [K_v:\Q_v] \log \max(\abs{x_0}_v,\dots,\abs{x_n}_v ) \]
dans laquelle $K$ est un corps de nombres,
$x$ est un point de~$\P^n(K)$ de coordonnées homogènes $[x_0:\dots:x_n]$ 
dans~$K$, $v$ parcourt l'ensemble des valeurs absolues sur~$K$
qui étendent la valeur absolue réelle ou une valeur absolue $p$-adique
sur~$\Q$, $K_v$ est le complété de~$K$ pour cette valeur absolue
et $\Q_v$ est le corps réel~$\R$ ou le corps $p$-adique~$\Q_p$
suivant les cas.
La formule du produit garantit que le second membre
ne dépend pas du choix des coordonnées homogènes ; il ne
dépend pas non plus du choix d'un corps de nombres~$K$
sur lequel le point~$x$ est défini.
Par exemple, si le point $x\in\P^n(\bQ)$ 
a pour coordonnées homogènes $[x_0:\dots:x_n]$
formée d'entiers relatifs premiers entre eux dans leur ensemble,
on a $h(x)=\log\max(\abs{x_0},\dots,\abs{x_n})$.
Remarquons aussi que la fonction~$h$ est invariante
sous l'action du groupe de Galois~$\Gal(\bQ/\Q)$.

Lorsque $X$ est une variété algébrique projective définie
sur~$\bQ$ et $\phi\colon X\ra\P^n$ est un morphisme de~$X$
dans un espace projectif, 
on en déduit une fonction $h_\phi=h\circ\phi$ sur~$X(\bQ)$, 
invariante sous l'action de~$\Gal(\bQ/K)$
si $X$ et~$\phi$ sont définis sur un sous-corps~$K$ de~$\bQ$.
Pour l'essentiel, la fonction~$h_\phi$ ne dépend de~$\phi$ que par
l'intermédiaire de la classe d'isomorphisme
du fibré en droites $\phi^*\mathscr O(1)$ dans
le groupe de Picard~$\Pic(X)$.
En effet, si $\psi$ est un morphisme
de~$X$ dans un espace projectif~$\P^m$ tel que $\phi^*\mathscr O(1)$
et $\psi^*\mathscr O(1)$ soient isomorphes, la différence
$h_\phi-h_\psi$ est une fonction bornée sur~$X(\bQ)$.

Notons $\mathscr F(X(\bQ),\R)$ l'espace des fonctions
de~$X(\bQ)$ dans~$\R$ et $\mathscr B$ son sous-espace
des fonctions bornées. 
Via l'étude des plongements de Segre ou Veronese, on démontre
qu'il existe un unique  morphisme d'espaces vectoriels
\[ \Pic(X)\otimes_\Z\R \ra \mathscr F(X(\bQ),\R)/ \mathscr B(X(\bQ),\R) \]
qui applique la classe d'un fibré en droites~$\mathscr L$
sur celle de la fonction~$h_\phi$ lorsque $\phi$
est un morphisme de~$X$ dans un espace projectif 
tel que $\phi^*\mathscr O(1)\simeq\mathscr L$.
Si $\mathscr L$
est un ($\R$-)fibré en droites sur~$X$, on notera abusivement
$h_{\mathscr L}$ une fonction de~$X(\bQ)$ dans~$\R$ qui représente
l'image de~$\mathscr L$ par ce morphisme de groupes.

Nous ferons usage des résultats suivants:
\begin{prop}\label{prop.heights}
Soit $X$ une variété projective définie sur~$\bQ$
et soit $\mathscr L$ un fibré en droites sur~$X$.
\begin{enumerate}
\item  La fonction 
$h_{\mathscr L}$ est minorée en dehors du support du diviseur
de toute section globale de~$\mathscr L$;
\item Si $\mathscr L$ est engendré par ses sections globales,
$h_{\mathscr L}$ est minorée ;
\item Supposons que $\mathscr L$ soit ample. Alors,
pour tout entier~$B$, l'ensemble des points~$x\in X(\bQ)$
tels que $[\Q(x):\Q]\leq B$ et $h_{\mathscr L}(x)\leq B$
est fini \emph{(théorème de \textsc{Northcott)}} ;
\item Soit $f\colon Y\ra X$ un morphisme de variétés algébriques
définies sur~$\bQ$, alors $h_{f^*\mathscr L}-h_{\mathscr L}\circ f$
est une fonction bornée sur~$Y(\bQ)$.
\end{enumerate}
\end{prop}

Pour plus de détails et des démonstrations,
je renvoie aux ouvrages d'introduction à la géométrie diophantienne, 
notamment~\cite{lang1983,serre1997,hindry-silverman2000,bombieri-gubler2006}.

Le point de vue de la théorie d'Arakelov (voir~\cite{bost-g-s94})
fournit un moyen efficace pour ne pas travailler à une fonction bornée près.
Dans le cas des variétés semi-abéliennes, il existe un moyen
simple pour \emph{normaliser} les hauteurs que je dois rappeler brièvement;
les hauteurs normalisées interviennent en effet de manière cruciale
dans l'étude des conjectures auxquelles ce rapport est consacré.

\subsection{Hauteurs normalisées sur les variétés semi-abéliennes}

Soit $G$ une variété semi-abélienne. Comme je l'ai rappelé plus haut,
$G$ est extension d'une variété  abélienne~$A$ par un tore~$T$:
\[ 1 \ra T \ra G \xrightarrow p A \ra 0 , \]
et (sur~$\bQ$, ou quitte à effectuer une extension finie)
$T$ est isomorphe à une puissance~$\gm^t$ du groupe multiplicatif.
Comme $G$ n'est que quasi-projectif (à moins que $T=\{1\}$),
les hauteurs dépendent d'une compactification projective de~$G$.
Lorsque $G=T=\gm^t$, on peut considérer $G$ comme un ouvert de
l'espace projectif~$\P^t$; toute autre compactification
équivariante~$P$ (disons projective, lisse) convient, par exemple $(\P^1)^t$.
Soit $\bar G$ le produit contracté $\bar G=G\times^T P$;
c'est une compactification équivariante de~$G$
munie d'une fibration vers~$A$,
toujours notée $p$, dont les fibres sont isomorphes à~$P$.
Pour $n\geq 2$, les endomorphismes de  multiplication par~$n$
sur~$G$ s'étendent en des endomorphismes de~$\bar G$.
En outre, le groupe de Picard de~$\bar G$ se décompose
(non canoniquement) en une somme directe
\[ \Pic(\bar G) \simeq \Pic(A) \oplus \Pic(P), \]
où $\Pic(A)$ est identifié à son image dans~$\Pic(\bar G)$ par
l'homomorphisme injectif~$p^*$.
Après tensorisation par~$\Q$,
le groupe de Picard de~$A$ se décompose 
sous l'action de l'endomorphisme~$[-1]$
en une partie paire et une partie impaire, provenant
respectivement du groupe de Néron--Severi de~$A$
et du groupe des classes d'isomorphisme de fibrés en droites
algébriquement équivalent à~$0$, d'où finalement
une décomposition
\[ \Pic(\bar G)_\Q \simeq \NS(A)_\Q \oplus \Pic^0(A)_\Q \oplus \Pic(P)_\Q.\]
Sous l'action des endomorphismes de multiplication par un entier~$n\geq 2$,
le premier facteur est de poids~$n^2$, tandis que les deux autres
sont de poids~$n$. Supposons que la classe de~$\mathscr L$
appartient à l'un de ces trois facteurs et posons respectivement
$w=2$, $w=1$, $w=1$.
On déduit alors du point~4 de la proposition~\ref{prop.heights}
que la fonction $x\mapsto h_{\mathscr L}([n]x)-n^w h_{\mathscr L}(x)$
est bornée sur~$\bar G(\bQ)$.
Par le procédé de \textsc{Tate} qui consiste à poser
\[ \hat h_{\mathscr L}(x) = \lim_{k\ra\infty} n^{-wk} h_{\mathscr L}([n]^k x),\]
on obtient une fonction $\hat h_{\mathscr L}$ sur~$\bar G(\bQ)$
telle que $\hat h_{\mathscr L}([n]x)=n^w \hat h_{\mathscr L}(x)$
pour tout $x\in \bar G(\bQ)$. Elle ne dépend pas du choix de~$n$.
En outre, la différence $\hat h_{\mathscr L}-h_{\mathscr L}$ est bornée
sur $\bar G(\bQ)$ : la fonction $\hat h_{\mathscr L}$
est appelée \emph{hauteur normalisée} pour le fibré en droites~$\mathscr L$.
Lorsque $G=A$ est une variété abélienne
et $\mathscr L$ est symétrique, on retrouve bien sûr
la forme quadratique de \textsc{Néron--Tate} ;
lorsque $G$ est le tore~$\gm^n$, ouvert de $P=\mathbf P^n$,
et $\mathscr L=\mathscr O(1)$,
$\hat h_{\mathscr L}$ est la hauteur standard.

Par additivité,
on en déduit 
un morphisme d'espaces vectoriels
\[ \Pic(\bar G)_\R \ra \mathscr F(\bar G(\bQ),\R),
\qquad  \mathscr L \mapsto \hat h_{\mathscr L} . \]

La proposition~\ref{prop.heights} s'étend facilement, seule la propriété~4
de fonctorialité requiert  un ajustement:
Sous l'hypothèse que $f\colon G'\ra G$ soit un morphisme
de variétés semi-abéliennes qui s'étend 
en un morphisme $\bar f\colon \bar G'\ra \bar G$ des compactifications
fixées, on a $\hat h_{\mathscr L}(\bar f(x))=\hat h_{\bar f^*\mathscr L}(x)$
pour tout $x\in \bar G'(\bQ)$.

Les points de torsion sont de hauteur normalisée nulle; inversement,
si $\mathscr L$ est ample, on déduit du théorème
de Northcott que les points~$x$ de~$G(\bQ)$
tels que $\hat h_{\mathscr L}(x)=0$ sont des points de torsion.

\subsection{Le problème de Lehmer}

En 1933, D.~H.~\textsc{Lehmer} demandait s'il existe, pour~$\eps>0$,
un polynôme~$P$,
unitaire à coefficients entiers, dont la mesure de Mahler~$\mathrm M(P)$
vérifie $1<\mathrm M(P)<1+\eps$; il ajoutait ne pas
savoir si ce problème a une solution pour $\eps<0{,}176$. 
En termes de hauteurs\footnote{La mesure
de Mahler d'un polynôme~$P$ est définie
par $\mathrm M(P)=\exp\left(\int_0^1 \log\abs{P(e^{2i\pi\theta})}\,\mathrm d\theta\right)$.
Si $P$ est le polynôme
minimal d'un nombre algébrique~$\xi$, on a $\mathrm M(P)=\exp(\deg(P) h(\xi))$;
les polynômes irréductibles de~$\Z[T]$ de mesure de Mahler nulle sont, 
outre le polynôme~$T$, les polynômes cyclotomiques.
Dans son article, \textsc{Lehmer} donne l'exemple du polynôme
minimal d'un nombre de Salem de degré~10 pour lequel
$\mathrm M(P)\approx1{,}1762$.},
la question est l'existence d'un nombre
réel~$c>0$ tel que 
pour tout nombre algébrique~$\xi\in\bQ^*$
qui n'est pas une racine de l'unité, on ait $h(\xi)\geq c/[\Q(\xi):\Q]$.
Une telle minoration serait optimale puisque $h(2^{1/n})=\frac 1n\log 2$
et $\deg(2^{1/n})=n$. En particulier, 
c'est le fait que la hauteur standard
d'un nombre algébrique vérifie une équation
fonctionnelle qui permet de formuler cette question
de façon pertinente.

Bien qu'elle soit encore ouverte, un théorème
de \textsc{Dobrowolski} résout cette question à~$\eps$ près
(et même un peu plus...):
\begin{theo}[\textsc{Dobrowolski} \cite{dobrowolski1979}]
Pour tout $\eps>0$, il existe un nombre réel $c(\eps)>0$
tel que pour tout nombre algébrique~$\xi$ qui n'est ni nul
ni une racine de l'unité, on ait
\[ h(\xi) \geq \frac {c(\eps)}{[\Q(\xi):\Q]^{1+\eps}}. \]
\end{theo}

La démonstration est astucieuse mais élémentaire, l'idée
principale consistant à exploiter les  congruences issues du
petit théorème de Fermat pour de nombreux nombres premiers~$p$.

\medskip

La question se pose, plus généralement, 
de fournir une minoration fine
de la \emph{hauteur normalisée}~$\hat h$ 
d'un point
d'une variété semi-abélienne, lorsque
ce point n'est pas de torsion.
Dans le cas torique ou abélien,
S.~\textsc{David}  a proposé la conjecture suivante.
Comme D.~\textsc{Bertrand} me l'a fait remarquer, le cas
des variétés semi-abéliennes générales pose des problèmes spécifiques.
\begin{conj}[Problème de Lehmer]
\label{conj.lehmer}
Soit $G$ un tore ou une variété abélienne sur~$\bQ$,
soit $g$ sa dimension. 
Soit $\hat h$ une hauteur normalisée associée à un fibré en droites
ample d'une compactification. 
Il existe un nombre réel~$c>0$
tel que pour tout point~$x\in G(\bQ)$ qui n'est contenu
dans aucun sous-groupe algébrique strict de~$G$,
on ait la minoration:
\[ \hat h(x) \geq c \, [\Q(x):\Q]^{-1/g}. \]
\end{conj}
Lorsque $G$ est définie sur un corps de nombres~$K$,
une conjecture plus précise remplace le degré $[\Q(x):\Q]^{1/g}$
par l'indice d'obstruction~$\omega_K(x)$, défini
comme le minimum des quantités $\deg(V)^{1/\codim(V)}$,
où $V$ est une sous-variété de~$G$ définie sur~$K$
contenant~$x$. Elle affirme l'existence d'un nombre réel~$c>0$
tel que si $x$ n'est pas de torsion et si $\hat h(x)$ est inférieur
à~$c\,\omega_K(x)^{-1}$, alors il existe un sous-groupe
algébrique~$H$ de~$G$ tel que $\deg(H^\circ)^{1/\codim(H)}\leq c \,\omega_K(x)$.

De fait, M.~\textsc{Laurent} avait démontré un tel résultat  à~$\eps$ près
pour les courbes elliptiques à multiplications
complexes dans~\cite{laurent:1983}.
Dans l'approche de \textsc{Dobrowolski}, 
pour transformer le petit théorème de Fermat
en une congruence, il faut en effet disposer,
pour beaucoup d'idéaux maximaux
de l'anneau des entiers du corps de base,
d'un endomorphisme
de la courbe elliptique qui relève l'endomorphisme
de Frobenius de sa réduction.
C'est précisément ce que permet la théorie de la multiplication complexe.

Dans des articles extrêmement délicats, F.~\textsc{Amoroso} et \textsc{David}
ont étendu cette approche au cas des tores, tandis
que \textsc{David} et M.~\textsc{Hindry} ont traité le cas 
des variétés abéliennes à multiplications complexes
(voir \cite{amoroso-david:1999}, \cite{david-hindry2000}).
Ainsi, quitte à remplacer l'exposant~$-1/g$ par $-1/g-\eps$,
la conjecture~\ref{conj.lehmer} est donc vraie
dans le cas des tores ou des variétés abéliennes à multiplications complexes.
Les preuves reposent sur les méthodes
d'approximation diophantienne, mais leur technicité
m'empêche d'en dire quoi que ce soit dans ce rapport.

\medskip

Dès le théorème initial de \textsc{Bombieri}, \textsc{Masser}
et \textsc{Zannier}, ces minorations de hauteurs ont joué un rôle
crucial dans la démonstration des théorèmes de finitude
auxquels ce rapport est consacré.
Toutefois, des raffinements récents permettent de simplifier leur
utilisation, voire d'en améliorer l'efficacité.
Présentons-les brièvement.

En 2000, \textsc{Amoroso} et R.~\textsc{Dvornicich}
ont démontré que la hauteur d'un élément~$\xi$ de l'extension
cyclotomique maximale de~$\Q$ est minorée par $\log(5)/12$,
à moins que $\xi$ ne soit nul ou une racine de l'unité.
Ce théorème a suscité toute une série d'articles
visant à remplacer, dans les conjectures de type Lehmer,
le degré sur~$\Q$, ou l'indice d'obstruction sur le corps de base~$K$,
par les quantités équivalentes sur le corps~$K_\tors$
engendré sur~$K$ par les coordonnées des points de torsion de~$G$.

\begin{conj}[Problème de Lehmer relatif]
\label{conj.lehmer-relatif}
Soit $G$ un tore ou une variété abélienne sur un corps
de nombres~$K$, soit $g$ sa dimension. 
Soit $K_\tors$ l'extension de~$K$ engendrée par les coordonnées
des points de torsion de~$G$.
Il existe un nombre réel~$c>0$
tel que pour tout point~$x\in G(\bQ)$ qui n'est contenu
dans aucun sous-groupe algébrique strict de~$G$,
on ait la minoration:
\[ \hat h(x) \geq c \, [K_\tors(x):K_\tors]^{-1/g}. \]
\end{conj}

Après divers travaux, les meilleurs résultats dans cette direction
sont dus à E.~\textsc{Delsinne} pour les tores
et à M.~\textsc{Carrizosa} pour les variétés
abéliennes.
\begin{theo}[\textsc{Delsinne}~\cite{delsinne:2009},
\textsc{Carrizosa}~\cite{carrizosa:2009}]
\label{theo.delsinne}
Soit $G$ une variété semi-abélienne définie sur un corps
de nombres~$K$, soit $g$ la dimension de~$G$
et soit $K_\tors$ l'extension de~$K$
engendrée par les points de torsion de~$G$.
Supposons que $G$ soit un tore 
ou une variété abélienne à multiplications complexes.
Alors, pour tout $\eps>0$, il existe un nombre réel~$c>0$
tel que pour tout point $x\in G(\bQ)$ qui n'appartient
à aucun sous-groupe algébrique strict de~$G$,
\[ h(x) \geq c\, [K_\tors(x):K_\tors]^{-1/g-\eps} . \]
\end{theo}

 
%
%
 
\subsection{La conjecture de Bogomolov}

La conjecture énoncée par F.~\textsc{Bogomolov} dans~\cite{bogomolov80b}
suggère un autre énoncé de minoration de la hauteur
d'un point d'une variété abélienne dans lequel la contrainte
n'est pas le corps de définition de ce point, mais son appartenance
à une sous-variété fixée.

Soit $G$ une variété semi-abélienne définie sur~$\bQ$; 
considérons un fibré en droites ample d'une compactification~$\bar G$
et les fonctions degré et hauteur canonique,
naturellement notées $\deg$ et $\hat h$, qui lui sont associées.

Soit $X$ une sous-variété (fermée, irréductible) de~$G$.
En considérant un morphisme fini de l'adhérence de~$X$
dans~$\bar G$ vers un espace projectif
associé à ce fibré en droites, on démontre aisément qu'il 
existe un nombre réel~$\theta$ tel que l'ensemble des points
de~$X(\bQ)$  de hauteur~$\leq\theta$ est dense dans~$X$
pour la topologie de Zariski.
Le \emph{minimum essentiel} de~$X$ est la borne inférieure
de ces nombres réels; on le note $\hat\mu(X)$.
Si $X$ est une sous-variété semi-abélienne de~$G$,
l'ensemble de ses points de torsion est dense dans~$X$
et $\hat\mu(X)=0$; c'est essentiellement la seule possibilité
pour que $\hat\mu(X)$ soit nul.

\begin{theo}\label{theo.bogomolov}
Si $X$ n'est pas le translaté d'une sous-variété semi-abélienne
de~$G$ par un point de torsion, on a $\hat\mu(X)>0$.
\end{theo}
Le cas des tores est dû à S.-W.~\textsc{Zhang}~\cite{zhang95}
(une preuve ultérieure, plus élémentaire, se trouve
dans~\cite{bombieri-zannier:1995,zannier:2009}).
Dans le cas des variétés abéliennes,
la démonstration de ce théorème  par E.~\textsc{Ullmo} (lorsque $X$
est une courbe dans sa jacobienne) et S.-W.~\textsc{Zhang} (en général)
repose sur des techniques d'équidistribution en géométrie d'Arakelov;
elle a été exposée dans ce séminaire (voir \cite{abbes1997}). 
Par une méthode plus proche de 
la géométrie diophantienne {\og traditionnelle\fg},
S.~\textsc{David}  et P.~\textsc{Philippon} ont redémontré
ces résultats et traité le cas général
des variétés semi-abéliennes (voir \cite{david-p98,david-p99,david-p2000}).

Ces dernières démonstrations 
ont en outre l'intérêt de fournir une minoration effective de~$\hat\mu(X)$.
Dans le cas où $X$ n'est pas le translaté d'une sous-variété
semi-abélienne de~$G$,
ces auteurs établissent en effet une minoration de $\hat\mu(X)$
inversement proportionnelle à une puissance du degré de~$X$.
(La minoration établie dans~\cite{bombieri-zannier:1995} pour les tores,
de moins bonne qualité, montrait comment déduire
une minoration géométrique du théorème ci-dessus.)
Dans cette optique,
un énoncé essentiellement optimal a été prouvé par
\textsc{Amoroso} et \textsc{David} dans~\cite{amoroso-david2003} dans le cas
des tores (voir aussi~\cite{amoroso-viada:2009});
les énoncés de ces références sont explicites,
je simplifie ici les termes logarithmiques: 
\begin{theo}
Soit $G$ un tore, $\bar G$ une compactification équivariante de~$G$,
$\mathscr L$ un fibré en droites ample sur~$\bar G$, $\deg$
et $\hat h$ les fonctions degré et hauteur normalisée associées.
Pour tout~$\eps>0$, 
il existe un nombre réel~$c$ (ne dépendant que de~$G$, $\bar G$, $\mathscr L$ et~$\eps$)
tel que l'on ait, pour toute sous-variété fermée $X\subsetneq G$,
l'inégalité
\[ \hat\mu(X) \geq  c\, \deg(X)^{-1/\codim(X)-\eps} ,\]
pourvu que $X$ ne soit contenue
dans aucun translaté de sous-groupe algébrique strict de~$G$.
\end{theo}

À~$\eps$ près, cet énoncé est optimal: 
en remplaçant $X$ par son inverse par la multiplication par
des entiers~$n\geq 2$, on voit que le meilleur
exposant possible du degré est l'opposé de l'inverse
de la codimension de~$X$ dans le plus petit sous-groupe
algébrique de~$G$ qui le contient.
Noter aussi que lorsque $X$ est un translaté d'un sous-groupe 
algébrique~$H$ de~$G$
par un point qui n'est pas de torsion modulo~$H$,
on a bien $\hat\mu(X)>0$, mais l'obtention d'une meilleure
borne relève du problème de Lehmer dans~$G/H$.

Il est naturel de poser une conjecture similaire dans
le cas des variétés abéliennes (voire des variétés semi-abéliennes):
\begin{conj}[Problème de Bogomolov effectif]
\label{conj.bogomolov}
Soit $G$ un tore ou une variété abélienne.
Soit $\hat h$ une hauteur normalisée associée à un fibré
en droites ample d'une compactification.
Il existe un nombre réel~$c>0$
tel que pour toute sous-variété (fermée, irréductible) $X$
de~$G$, distincte de~$G$, 
qui n'est pas contenue dans un translaté de sous-groupe algébrique
strict de~$G$, on ait la minoration
\[ \hat\mu(X) \geq c\, (\deg(X))^{-1/\codim(X)}. \]
\end{conj}

A.~\textsc{Galateau} a fait de grands progrès dans cette direction;
l'intérêt pour les questions présentées dans ce rapport est qu'il dépasse
le strict cadre des variétés abéliennes  à multiplications complexes.

Le théorème de \textsc{Galateau} est alors le suivant:
\begin{theo}[\cite{galateau:2010}]
\label{theo.galateau}
Soit $G$ une variété abélienne définie sur~$\bQ$,
$\mathscr L$ un fibré en droites ample sur~$G$,
$\deg$ et $\hat h$ les fonctions degré et hauteur
normalisée associées.
On suppose vérifiée l'hypothèse suivante:
\begin{quote}
Il existe un corps de nombres~$K$ sur lequel $G$
est définie et tel que l'ensemble des idéaux premiers~$\mathfrak p$
de l'anneau des entiers de~$K$ en lequel $G$ a bonne réduction
ordinaire soit de densité strictement positive.
\end{quote}
Pour tout~$\eps>0$,
il existe un nombre réel~$c>0$ (ne dépendant que de~$G$, $\mathscr L$,
$\eps$)
tel que 
\[ \hat\mu(X) > c\, (\deg(X))^{-1/\codim(X)-\eps} \]
pour toute sous-variété fermée irréductible~$X$ de~$G$,
distincte de~$G$
qui n'est pas contenue dans un translaté de sous-variété abélienne stricte.
\end{theo}

Rappelons qu'on dit que $G$ 
a bonne réduction en l'idéal premier~$\mathfrak p$
de~$K$ s'il existe un schéma abélien~$\mathscr G$ sur
l'anneau $\mathfrak o_{K,\mathfrak p}$ qui étend~$G$;
sa réduction est alors la variété abélienne 
$G_{\mathfrak p}=\mathscr G\otimes \F_{\mathfrak p}$ sur le corps
résiduel $\F_{\mathfrak p}=\mathfrak o_{K,\mathfrak p}/\mathfrak p$.
Elle a de plus bonne réduction ordinaire si,
dans une clôture algébrique de~$\F_{\mathfrak p}$,
le groupe des points de $p$-torsion de~$G_{\mathfrak p}$
est de cardinal $p^{\dim(G)}$, où l'on a noté $p$ la caractéristique
du corps fini~$\F_{\mathfrak p}$. 

Enfin, la densité considérée est la densité naturelle.

Pour évaluer le statut de l'hypothèse faite dans le théorème
précédent, posons une définition.

\begin{defi}
Soit $G$ une variété abélienne définie sur~$\bQ$.
On dit que~$G$ est banale s'il existe un corps de nombres~$K$
sur lequel $G$ est définie et tel que l'ensemble des idéaux premiers de~$K$
en lesquels $G$ ait bonne réduction \emph{ordinaire}
soit de densité égale à~$1$.
\end{defi}
Un produit de variétés abéliennes banales,
un quotient d'une variété abélienne banale sont banals.
En outre, 
pour une variété abélienne sur un corps de caractéristique~\mbox{$p>0$},
être ordinaire est une propriété générique dans l'espace des modules
des variétés abéliennes. 
Il est ainsi conjecturé que
toute variété abélienne~$G$ définie sur~$\bQ$
est banale (\cf~\textsc{Pink}~\cite{pink:1998}, \S7).
Dans cette direction, on a les résultats suivants:

\begin{prop}
Soit $G$ une variété abélienne  définie sur~$\bQ$.
Si $\dim(G)\leq 2$ ou si $G$ est une variété abélienne à multiplications
complexes, alors $G$ est banale.
\end{prop}
Le cas des courbes elliptiques est dû à J-P.~\textsc{Serre}~\cite{serre68b},
celui des surfaces abéliennes à A.~\textsc{Ogus}~\cite{ogus82},
et le cas des variétés abéliennes à multiplications complexes
résulte de cette théorie. Lorsque $G$ est définie
sur un corps de nombres~$K$, R.~\textsc{Noot}
et \textsc{Pink} ont
aussi donné des conditions suffisantes sur l'action
du groupe de Galois $\Gal(\bQ/K)$ sur les modules de Tate de~$G$
assurant la conclusion de la proposition.

\section{Théorèmes de finitude}

\subsection{Majoration de la hauteur
hors de l'ensemble exceptionnel}

Soit $G$ une variété semi-abélienne et soit $X$ une sous-variété
(fermée, irréductible) de~$G$.
Pour prouver la finitude de l'ensemble
$\Sigma=X^{\ta}(\bQ)\cap G^{[1+\dim(X)]}$,
l'approche inaugurée par \textsc{Bombieri}, \textsc{Masser} et \textsc{Zannier} 
dans leur article~\cite{bombieri-masser-zannier:1999}
fonctionne en deux étapes.

La première étape, peut-être la plus délicate, 
consiste à prouver que l'ensemble~$\Sigma$
est de hauteur bornée.
La seconde utilise des minorations de hauteurs
dans l'esprit de la conjecture de {Lehmer}
pour en déduire que le degré du corps
de définition des points de~$\Sigma$ est uniformément borné.
Le théorème de \textsc{Northcott}  entraîne alors
que $\Sigma$ est fini.

C'est à cette première étape qu'est consacré ce paragraphe.
Le cas des courbes possède une solution assez simple mais 
l'étude de la dimension supérieure s'avère plus délicate.
Après de nombreux résultats partiels dans cette direction, 
\textsc{Habegger} a finalement démontré le résultat suivant.
(Dans le cas des tores, il s'agit 
de la \emph{Bounded Height Conjecture} de~\cite{bombieri-masser-zannier:2007}.)

\begin{theo}[\textsc{Habegger}, \cite{habegger:2009a,habegger:2009d}]
\label{theo.habegger}
Soit $G$ un tore ou une variété abélienne et soit
$X$ une sous-variété (fermée, irréductible) de~$G$.
L'ensemble $X^\oa(\bQ)\cap G^{[\dim(X)]}$ est de hauteur bornée.
\end{theo}

Commençons par deux remarques sur le caractère
essentiellement optimal de cet énoncé
(voir~\cite{habegger:2008,viada:2009b}).
\begin{rema}
\emph{L'ensemble $X^\oa(\bQ)\cap G^{[\dim(X)]}$ 
peut être dense dans~$X$.} 

On l'a par exemple vu dans le cas d'une courbe de~$\gm^n$,
un peu après l'énoncé du théorème~\ref{theo.maurin}.
\end{rema}

\begin{rema}
\emph{Si $Y$ est  une
composante irréductible de~$X\setminus X^\oa$,
il n'existe pas forcément d'ouvert dense~$U$ de~$Y$
tel que la hauteur soit bornée  sur $U(\bQ)\cap G^{[\dim(X)]}$.}

Contentons-nous de traiter le cas d'une courbe~$X$ dans~$G=\gm^n$
telle que $X=X^\oa$.
Dans ce cas, nous devons démontrer qu'exiger une relation
de dépendance multiplicative non triviale entre les coordonnées
d'un point de~$X$ ne suffit pas à majorer sa hauteur. 
Par un changement de coordonnées
sur~$\gm^n$, on se ramène en effet au cas où les restrictions
à~$X$ des $m$ premières coordonnées sont multiplicativement indépendantes  
et celles des $n-m>0$ dernières sont constantes 
de valeurs~$\xi_{m+1},\dots,\xi_n$.
Si $\xi_n$ est une racine de l'unité, $X$ est contenu
dans un sous-groupe algébrique d'équation $x_n^e=1$,
tout point de~$X$ satisfait une relation de dépendance
multiplicative non triviale et la hauteur n'est bornée
sur aucun ouvert non vide de~$X$. Sinon, 
pour presque tout entier naturel~$a$,
l'intersection de~$X$ avec le sous-tore d'équation $x_1=x_n^a$
n'est pas vide et contient un point~$P_a$ dont la hauteur
vérifie $h(P_a)\geq h(x_1(P_a))=h(x_n(P_a)^a)=ah(\xi_n)$,
donc tend vers l'infini avec~$a$.
\end{rema}

Expliquons maintenant la preuve 
du théorème~\ref{theo.habegger}
dans le cas des tores, suivant~\cite{habegger:2009d};
je renvoie à~\cite{habegger:2009a} pour celle, voisine,
du cas abélien.

Si $m$ et $n$ sont des entiers naturels, rappelons
qu'un morphisme~$\phi$ de~$\gm^n$ dans~$\gm^m$ est donné
par $m$~monômes en $n$~variables; on identifie alors~$\phi$
à la matrice~$M_\phi$ de taille~$m\times n$ formée par les exposants
de ces monômes. 
On notera aussi $\norm\phi$ la norme euclidienne de cette matrice.

Posons $m=\dim(X)$.
Si $\phi$ est un morphisme de~$\gm^n$ dans~$\gm^m$,
notons $\Delta_X(\phi)$ le degré générique
du morphisme $\phi|_X\colon X\ra\gm^m$. 
Il est homogène de degré~$m$ en la matrice de~$\phi$:
pour tout entier~$a\in\Z$, $\Delta_X(\phi^a)=\abs a^m\Delta_X(\phi)$.
Il se calcule aussi à l'aide de la théorie de l'intersection.
Soit en effet~$X_\phi$ l'adhérence du graphe de~$\phi$
dans~$\P^n\times\P^m$ et notons $p_1$, $p_2$ les deux projections
de~$X_\phi$ sur~$\P^n$ et~$\P^m$.
Alors, 
\begin{equation}
\Delta_X(\phi) = \deg(c_1(p_2^*\mathscr O(1))^m\cap [X_\phi]). 
\end{equation}

\begin{lemm}\label{lemm.minoration}
Pour tout~$\phi\colon\gm^n\ra\gm^m$, il existe 
un nombre réel~$c(\phi)$ et un ouvert dense~$U_\phi$ de~$X$
tel que l'on ait, pour tout $x\in U_\phi(\bQ)$,
\begin{equation}
 h(\phi(x)) \gg \frac{\Delta_X(\phi)}{\norm{\phi}^{m-1}} h(x) - c(\phi), 
\end{equation}
où la constante implicite dans le symbole~$\gg$ est indépendante de~$\phi$.\footnote{%
Si $u$ et~$v$ sont deux fonctions, j'utilise la notation $u\gg v$
pour dire qu'il existe un nombre réel~$c>0$
tel que $u\geq cv$; si $w$ est un paramètre, $u\gg_w v$
signifie que pour tout~$w$, il existe un nombre réel~$c_w>0$
tel que $u\geq c_w v$.}

\end{lemm}
\begin{proof}
Soit $t\in\Q$ et soit $\mathscr L_t$ le $\Q$-fibré en droites\footnote{Je
note additivement la loi de groupes sur le groupe de Picard
déduite du produit tensoriel.}
$p_2^*\mathscr O(1)-t\,p_1^*\mathscr O(1)$ sur~$X_\phi$. 
Il suffit, pour démontrer ce lemme, d'établir que $\mathscr L_t$
est $\Q$-effectif pour une valeur de~$t$ 
de l'ordre de~$\Delta_X(\phi)/\norm{\phi}^{m-1}$.
D'après un théorème de \textsc{Siu}~\cite{siu1993},
c'est le cas dès que l'inégalité
\[ \deg(c_1(p_2^*\mathscr O(1))^m\cap [X_\phi])
> m \, t \, \deg(c_1(p_1^*\mathscr O(1))c_1(p_2^*\mathscr O(1))^{m-1}\cap [X_\phi])
\]
est satisfaite.
Grâce à une inégalité de type Bézout due à P.~\textsc{Philippon},
le degré d'intersection du membre de droite est
majoré par un multiple de~$\norm{\phi}^{m-1}$,
ce qui permet de conclure.
\end{proof}

\begin{lemm}\label{lemm.majoration}
Soit~$\Omega$ un voisinage de l'ensemble des matrices orthogonales\footnote{J'entends par là que leurs lignes forment une famille orthonormée.}
dans~$\mathrm M_{m,n}(\R)$. Il existe un nombre réel~$Q_0$
tel que pour tout nombre réel~$Q>Q_0$, l'assertion suivante soit satisfaite:
Si $x\in\gm^n(\bQ)$ appartient à un sous-groupe de codimension~$\geq m$,
il existe un entier~$q\in\{1,\dots,Q\}$ et un homomorphisme 
$\phi\colon \gm^n\ra\gm^m$ tel que $\phi\in q\Omega$, $\norm{\phi}\ll q$
\begin{equation} h(\phi(x)) \ll q Q^{-1/mn} h(x). \end{equation}
\end{lemm}
\begin{proof}
Par hypothèse, il existe un morphisme surjectif~$\phi_0\colon\gm^n\ra\gm^m$
tel que $\phi_0(x)=1$. Par le procédé d'orthogonalisation de Gram--Schmidt, 
on écrit $\phi_0=\theta_1\phi_1$, où $\theta_1\in\GL_m(\R)$
et où $\phi_2\in \mathrm M_{m,n}(\R)$ est orthogonale.
Soit $\theta_2$ une matrice à coefficients rationnels 
assez proche de~$\theta_1$, de sorte que $\theta_2^{-1}\phi_0$
appartienne à~$\Omega$. D'après le lemme de Dirichlet,
on peut alors approcher cette dernière matrice
par une matrice à coefficients rationnels de la forme $q^{-1}\phi$
appartenant à~$\Omega$
telle que $\norm{q^{-1}\phi-\theta_2^{-1}\phi_0}\ll Q^{-1/mn}$,
où $1\leq q\leq Q$ et $\phi\in\mathrm M_{m,n}(\Z)$.
Ces inégalités entraînent que 
\[ \abs{ h(\phi(x))-h(\theta_2^{-1}\phi_0(x)^q) }\ll qQ^{-1/mn} h(x), \]
d'où le lemme puisque $\phi_0(x)=1$.
\end{proof}

Soit $\phi\colon\gm^n\ra\gm^m$ un morphisme surjectif de groupes algébriques.
Si $X^\oa\neq\emptyset$, 
observons que l'application~$\phi|_X$ est génériquement quasi-finie.
En effet, dans le cas contraire, il passerait par tout point de~$X$
une composante~$Y$ de dimension~$>0$ de~$X\cap\ker(\phi)$ ;
puisque $\dim\ker(\phi)=\dim (X)$, $Y$ est alors une composante
atypique de~$X$, ce qui entraîne~$X^\oa=\emptyset$, d'où une contradiction.
On a donc $\Delta_X(\phi)>0$.

La proposition suivante est une minoration uniforme.
Si $r$ est un entier naturel tel que~$r<m$, on appelle
projection standard de~$\gm^m$ sur~$\gm^r$ un morphisme de groupes
algébriques défini par l'oubli de $m-r$ coordonnées.
Je renvoie à~\cite{habegger:2009d}, \S6--7, pour sa démonstration.

\begin{prop}\label{prop.degres}
Soit $Y$ une sous-variété fermée, irréductible de~$X$
telle que $Y\cap X^\oa\neq\emptyset$, soit $r=\dim (Y)$.
Il existe un nombre réel~$c>0$ et un voisinage~$\Omega$ de l'ensemble
des matrices orthogonales dans~$\mathrm M_{m,n}(\R)$
tels que pour tout~$\phi\in\Omega$, il existe une projection
standard $\pi\colon\gm^m\ra\gm^r$ telle que $\Delta_Y(\pi\phi)\geq c$.
\end{prop}

On peut alors conclure la démonstration du théorème~\ref{theo.habegger}.
Soit $Q$ un entier naturel assez grand.
Appliquons la proposition~\ref{prop.degres} à $Y=X$;
soit $c>0$ tel que $\Delta_X(\phi)\geq c$ pour tout $\phi$
dans un voisinage~$\Omega$ de l'ensemble des matrices orthogonales
de taille~$m\times  n$.
Soit $x$ un point de~$X(\bQ)\cap G^{[m]}$.
Soit $q$ un entier et soit $\phi$ un morphisme de~$\gm^n$ dans~$\gm^m$
comme dans le lemme~\ref{lemm.majoration} ; on a
\[ h(\phi(x))\ll q Q^{-1/mn} h(x). \]
L'entier~$Q$ étant fixé, l'ensemble des couples~$(q,\phi)$ 
que peut fournir le lemme~\ref{lemm.majoration}
est fini ; le lemme~\ref{lemm.minoration}
implique donc l'existence d'un ouvert dense~$U$ de~$X$
et d'un nombre réel~$c(Q)$
tel que, si $x\in U$,
\[ h(\phi(x))\gg 
 \frac{\Delta_X(\phi)}{\norm{\phi}^{m-1}} h(x) - c(Q). \]
Comme $\phi\in q\Omega$,
\[ \frac{\Delta_X(\phi)}{\norm{\phi}^{m-1}}
=  \frac{q^m\Delta_X(q^{-1}\phi)}{\norm{\phi}^{m-1}}
\geq q c .  \]

Mises bout à bout, ces inégalités entraînent l'existence
d'un nombre  réel~$a>0$ et, pour tout entier~$Q$ assez grand,
d'un nombre réel~$c_Q$ et d'un ouvert dense~$U_Q$ de~$X$
tel que tout point~$x\in U_Q(\bQ)\cap G^{[m]}$
vérifie
\[   c h(x) - c_Q \leq a  Q^{-1/mn} h(x), \]
d'où $h(x)\leq c_Q/(c-a Q^{-1/mn})$ si $Q$ est assez grand.

Cela fournit la majoration de hauteurs souhaitée sur un ouvert
dense~$U$ de~$X^\oa$. En considérant les composantes irréductibles
de~$X\setminus U$, un argument de récurrence descendante
l'entraîne alors sur~$X^\oa$ tout entier.

\medskip

En fait, \textsc{Habegger} démontre un théorème plus général
où interviennent des épaississements de l'ensemble $G^{[\dim(X)]}$
au sens de la théorie des hauteurs.
Supposons donc fixé une hauteur canonique~$\hat h$ sur~$G$,
associée à un fibré ample d'une  compactification équivariante.
Pour toute partie~$\Sigma$ de~$G(\bQ)$,
notons $\mathscr C(\Sigma,\eps)$
l'ensemble\footnote{La lettre~$\mathscr C$ est l'initiale de cône.}
des points de~$G(\bQ)$ qui s'écrivent sous
la forme $xy$, où $x\in\Sigma$ et $\hat h(y)\leq \eps\max(1,\hat h(x))$.

Avec ces notations, on a alors le théorème:
\begin{theo}\label{theo.habegger-epsilon}
Soit $G$ un tore ou une variété abélienne et soit
$X$ une sous-variété (fermée, irréductible) de~$G$.
Il existe un nombre réel~$\eps>0$
tel que l'ensemble $X^\oa(\bQ)\cap \mathscr C(G^{[\dim(X)]},\eps)$
soit de hauteur bornée.
\end{theo}
La preuve de ce théorème s'établit de la même
façon que celle du théorème~\ref{theo.habegger},
avec quelques modifications consistant essentiellement à majorer,
lorsque $x\in\Sigma$ et $\hat h(y)\leq\eps(\max(1,\hat h(x)))$,
la hauteur de~$x$ en fonction de celle de~$xy$.

\subsection{Finitude}

Nous commençons par traiter le cas des tores.

\begin{prop}[\cite{bombieri-masser-zannier:2008}, Lemme~8.1]
\label{prop.fini}
Soit $X$ une sous-variété fermée, irréductible,
stricte d'un tore~$G=\gm^n$. 
Alors, pour tout nombre réel~$B$,
l'ensemble des points~$x\in X^\ta(\bar\Q)\cap G^{[1+\dim(X)]}$
tels que $h(x)\leq B$ est fini.
\end{prop}
\begin{rema}
Pour obtenir un tel énoncé de finitude, il est nécessaire
de considérer l'ensemble~$X^\ta$. Considérons en effet
une composante irréductible~$Y$ 
d'une intersection~$X\cap H$,
où $H$ est un sous-groupe algébrique strict de~$G$;
supposons qu'elle soit atypique et de dimension positive.
En considérant l'image réciproque des points de torsion
par une projection de~$H$ sur un tore~$\gm^{\dim(Y)}$
dont la restriction à~$Y$ est génériquement finie, on  
voit que $Y(\bQ)\cap H^{[\dim(Y)]}$ contient un ensemble
de hauteur bornée qui est dense dans~$Y$.
Par suite, si $B$ est assez grand, l'adhérence de l'ensemble des points
de~$X(\bQ)\cap G^{[1+\dim(X)]}$ dont la hauteur est~$\leq B$
contient~$Y$.
\end{rema}

\begin{proof}
Pour simplifier les notations, on note $m=\dim(X)$. 
Notons aussi $K$ un corps de définition de~$X$.
On définit la norme $\norm{\cdot}$ d'un vecteur 
comme le maximum de ses coordonnées.

Soit $x=(\xi_1,\dots,\xi_n)\in X^\ta(\bar\Q)\cap G^{[1+m]}$.
Soit $r$ la dimension du plus petit sous-groupe algébrique~$T_x$
qui contient~$x$: c'est le rang du sous-groupe multiplicatif~$\Gamma_x$
de~$\bQ^*$ engendré par $(\xi_1,\dots,\xi_n)$.
Par hypothèse, $\codim(T_x)\geq m+1$, c'est-à-dire $r+m+1\leq n$.

Un théorème de \textsc{Schlickewei}~\cite{schlickewei:1997}
fournit des éléments $\eta_1,\dots,\eta_r\in\Q(x)$
vérifiant les conditions suivantes:
\begin{itemize}
\item Il existe
des entiers relatifs $a_{ij}$ (pour $1\leq i\leq n$ et $1\leq j\leq r$)
et des racines de l'unité $\zeta_1,\dots,\zeta_n$ dans~$\Q(x)$
tels que $\xi_i=\zeta_i \eta_1^{a_{i1}}\dots \eta_r^{a_{ir}}$
pour tout $i$ tel que $1\leq i\leq n$;
\item Pour toute famille $(e_1,\dots,e_r)\in\Z^r$, 
$h(\eta_1^{e_1}\dots\eta_r^{e_r})\geq c(r) \sum_{j=1}^r \abs {e_j}h(\eta_j)$.
\end{itemize}
La preuve de cet énoncé consiste à introduire 
l'espace vectoriel~$\Gamma_x\otimes\R$
(isomorphe à~$\R^r$) muni de la jauge~$\omega$ donnée par la hauteur 
(on étend~$h$ à $\Gamma_x\otimes\Q$ par linéarité, puis à~$\Gamma_x\otimes\R$
par continuité). On remplace alors cette jauge par une autre,
quadratique, 
correspondant à l'ellipsoïde de John de la boule $\{\omega(\xi)\leq 1\}$.
La famille~$(\eta_1,\dots,\eta_r)$ provient d'une base LLL réduite
du réseau $\Gamma_x$ modulo torsion.


On retient en particulier les inégalités
\[
 h(\xi_i) \gg \sum_{j=1}^r \abs{a_{ij}} h(\eta_j)
, \qquad \text{pour $1\leq i\leq n$.} 
\]
Puisqu'on ne considère que des points~$x$ de hauteur au plus~$B$
et que $h$ est positive ou nulle, il vient en particulier
\begin{equation}
\label{eq.ajh(etaj)}
 \norm{a_j} h(\eta_j) \ll 1
, \qquad \text{pour $1\leq j\leq r$,} 
\end{equation}
où l'on a posé
$a_j=(a_{1j},\dots,a_{nj})$.

Soit $L$ le ppcm des ordres des~$\zeta_i$
et soit $\zeta\in\Q(x)$ une racine de l'unité d'ordre~$L$;
pour tout~$i$, soit $\ell_i$ un entier tel que $0\leq\ell_i<L$
et $\zeta_i=\zeta^{\ell_i}$. Considérons alors les $r+1$
formes linéaires indépendantes sur~$\Z^{n+1}$:
\[  \phi_0(\mathbf x)= -L x_0 + \sum_{i=1}^n \ell_i x_i,
\quad\text{et}\quad \phi_j(\mathbf x)=\sum_{i=1}^n a_{i,j} x_i 
\qquad \text{(pour $1\leq j\leq r$)}. \]
Un argument de géométrie des nombres (par exemple, le lemme de Siegel
de \textsc{Bombieri--Vaaler}) garantit l'existence
d'éléments $b_1,\dots,b_{n-r}\in\Z^{n+1}$, linéairement indépendants,
satisfaisant les relations $\phi_j(b_k)=0$ pour tout~$j$ et tout~$k$,
et de tailles contrôlées:
\[ \prod_{k=1}^{n-r} \norm{b_k} \ll L \prod_{j=1}^r \norm{a_j}. \]
Pour simplifier les notations, posons $A=\prod_{j=1}^r \norm{a_j}$.
On supposera aussi, ce qui est loisible,
que $\norm{b_1}\leq\dots\leq \norm{b_{n-r}}$.

Pour $1\leq k\leq n-r$, notons $(b_{k0},\dots,b_{kn})$
les coordonnées de~$b_k$.
Les caractères $\mathbf x \mapsto \prod_{i=1}^n x_i^{b_{ki}}$ sur~$\gm^n$,
pour $1\leq k\leq m$,
sont indépendants et définissent un sous-groupe algébrique~$T_b$
de codimension~$m$ dans~$\gm^n$. 
D'après le théorème de Bézout,
son degré (calculé dans l'espace projectif $\P^n$) est majoré
par
\[ \deg(T_b) \ll \prod_{k=1}^m \norm{b_k} \ll (LA)^{m/( n-r)}
\ll (LA)^{m/(m+1)} . \]
Par construction, on a $\prod_{i=1}^n \xi_i^{b_{ki}}=1$
pour tout $k\in\{1,\dots,n-r\}$; autrement dit, $x\in T_b$.

Considérons maintenant la composante $Y_x$ passant par~$x$
de l'intersection $X\cap T_b$. Comme $x\in X^{\ta}$,
cette composante est de dimension typique ou de dimension nulle.
Puisque $\dim(X)+\dim(T_b)-\dim(G)=0$, on a nécessairement
$Y_x=\{x\}$. Une nouvelle application du théorème de Bézout
dans l'espace projectif~$\P^n$
entraîne que le nombre de composantes ponctuelles de $X\cap T_b$
est au plus égal à $\deg(X)\deg(T_b)\ll (LA)^{m/(m+1)}$.
Comme tous les conjugués de~$x$ sur le corps de définition~$K$
de~$X$ appartiennent à $X\cap T_b$, on en déduit l'inégalité
\begin{equation}
\label{eq.degre(x)}
[ \Q(x):\Q ] \ll (LA)^{m/(m+1)}.
\end{equation}

C'est à ce stade qu'entrent en jeu les minorations effectives
concernant le problème de Lehmer dans les tores:
elles permettent de déduire de cette majoration du degré de~$x$, 
donc du point~$\eta$, une minoration  de la hauteur des~$\eta_j$.
Les éléments $(\eta_1,\dots,\eta_r)$ de~$\Q(x)^*$ sont multiplicativement
indépendants et $\zeta$ est une racine de l'unité contenue dans~$\Q(x)$.
Le théorème~\ref{theo.delsinne}
entraîne donc, pour tout $\eps>0$, une inégalité
\[ h(\eta_1)\dots h(\eta_r) \gg_\eps [ \Q(\eta):\Q(\zeta) ]^{-1-\eps}. \]
Puisque $\zeta$ est d'ordre~$L$, $[\Q(\zeta):\Q]=\phi(L)\gg_\eps L^{1-\eps}$.
On en déduit alors la minoration
\begin{equation}
  A \, h(\eta_1)\dots h(\eta_r) \gg_\eps  (LA)^{(1-m\eps)/(m+1)} .
\end{equation}

Si nous comparons cette inégalité à l'inégalité~\eqref{eq.ajh(etaj)},
on obtient
\[ LA \ll_\eps 1, \]
ce qui signifie que le produit~$LA$ est borné indépendamment de~$x$.
L'inégalité~\eqref{eq.degre(x)} entraîne alors que
lorsque $x$ parcourt l'ensemble qui nous intéresse,
$[\Q(x):\Q]$ est borné.
La proposition découle alors du théorème de \textsc{Northcott}. 
\end{proof}

Signalons que dans cette approche, la pleine force 
du théorème~\ref{theo.delsinne}
de \textsc{Delsinne} n'est pas vraiment nécessaire
mais simplifie grandement les démonstrations
en court-circuitant, par exemple,
les arguments  galoisiens qui avaient permis 
à \textsc{Bombieri}, \textsc{Masser} et \textsc{Zannier}
de prouver la proposition~\ref{prop.fini}
(voir~\cite{bombieri-masser-zannier:1999}, \S4,
ainsi que~\cite{bombieri-masser-zannier:2006}, Lemme~1 p.~2249)
lorsque $X$ est une courbe.
Dans~\cite{bombieri-masser-zannier:2008}, ils utilisent
la borne un peu plus faible de~\cite{amoroso-david2004}.

Compte tenu du théorème~\ref{theo.habegger} d'\textsc{Habegger}, on a
donc le théorème suivant:
\begin{coro}[\cite{habegger:2009b}, Corollaire~1.4]
Soit $X$ une sous-variété (fermée, irréductible) du tore $G=\gm^n$.
L'ensemble $X^\oa(\bQ)\cap G^{[1+\dim (X)]}$ est fini.
\end{coro}
Cela démontre en particulier le théorème~\ref{theo.nd}
dans le cas des tores:
si $X$ n'est pas géométriquement dégénérée, $X\setminus X^\oa$ est une
partie fermée stricte de~$X$  et
$X(\bQ)\cap G^{[1+\dim(X)]}$ n'est pas dense dans~$X$
pour la topologie de Zariski.

\medskip

L'analogue abélien de la proposition~\ref{prop.fini}
a été démontré par~\textsc{Rémond} dans~\cite{remond:2005b}
(théorème~2.1)
en supposant la conjecture~\ref{conj.lehmer-relatif} vérifiée.
Compte tenu du théorème~\ref{theo.delsinne} de \textsc{Carrizosa},
on a donc:
\begin{prop}\label{prop.fini-abelien}
 Soit $G$ une variété abélienne à multiplications complexes 
(définie sur~$\bQ$), soit $X$ une sous-variété fermée, irréductible,
de~$A$ définie sur~$\bQ$.
Alors, pour tout nombre réel~$B$,
l'ensemble des points~$x\in X^\ta(\bQ)\cap G^{[\dim(X)]}$ 
tels que $h(x)\leq B$ est fini.
En particulier, l'ensemble $X^\oa(\bQ)\cap G^{[1+\dim(X)]}$ est fini.
\end{prop}
La démonstration, bien que du même esprit que celle 
de la proposition~\ref{prop.fini}, en diffère par plusieurs
points. Des arguments de géométrie des nombres 
et le théorème de \textsc{Carrizosa}
permettent de prouver l'existence d'une famille finie
de sous-variétés abéliennes de codimension~$\geq\dim(X)$
telles que tout point~$x$ comme dans la proposition appartienne,
modulo un point de torsion, à l'une d'entre elles. \textsc{Rémond}
conclut alors à l'aide du théorème de \textsc{Raynaud}
sur la conjecture de Manin--Mumford.

Signalons aussi que \textsc{Rémond} démontre dans
la même proposition un théorème de finitude
analogue pour une intersection $X^{\oa,[s]}\cap G^{[t]}$
valable pour toute variété abélienne,
où $s\leq t$ sont des entiers convenables (mais on a toujours $t>s$
lorsqu'on doit se contenter de résultats
partiels en direction de la conjecture~\ref{conj.lehmer-relatif}).

Signalons encore que \textsc{Carrizosa} a raffiné le cas
abélien du théorème~\ref{theo.delsinne} 
(voir le théorème~1.15 de~\cite{carrizosa:2009});
cela lui permet de simplifier 
la preuve de la proposition~\ref{prop.fini-abelien}.

Enfin, et de manière analogue au cas des tores,
la version abélienne du théorème~\ref{theo.nd} 
est une conséquence
directe du théorème~\ref{theo.habegger} et de cette proposition.

\subsection{Quand l'exception est la règle}

Les théorèmes précédents ne donnent un théorème de finitude
que dans l'ouvert~$X^\oa$. En particulier, ils sont
vides quand $X^\oa$ l'est, c'est-à-dire quand $X$ est géométriquement
dégénérée.
La démonstration du théorème~\ref{theo.maurin'} requiert donc
des arguments supplémentaires.

Nous décrirons au paragraphe suivant l'approche de \textsc{Rémond}
vers la conjecture~\ref{conj.pink-ml} et la façon dont 
elle permit à \textsc{Maurin} de prouver ce théorème.
Expliquons pour l'instant 
comment \textsc{Bombieri}, \textsc{Habegger}, 
\textsc{Masser} et \textsc{Zannier} utilisent
le théorème~\ref{theo.habegger} pour y parvenir.

Nous considérons donc une courbe (irréductible, fermée)~$X$ du tore~$G=\gm^n$,
définie sur un corps de nombres~$K$, qui n'est pas
contenue dans un sous-groupe algébrique strict de~$\gm^n$.
Il s'agit de démontrer que $X(\bQ)\cap  G^{[2]}$ est fini.
On raisonne par récurrence sur~$n$; le cas $n<2$ est trivial
et le cas $n=2$ résulte de la finitude
de l'ensemble des points de torsion situés sur~$X$.

Lorsque $X$ n'est pas dégénérée, c'est-à-dire n'est contenue dans 
aucun translaté de sous-groupe algébrique strict de~$\gm^n$,
le résultat est couvert par le théorème~\ref{theo.nd}, 
prouvé dans ce cas par~\cite{bombieri-masser-zannier:1999}.
Supposons donc que $X$ soit contenue dans un translaté
d'un sous-groupe algébrique strict.
Au prix d'un changement de coordonnées sur~$\gm^n$, 
il existe un entier~$m\in\{1,\dots,n-1\}$,
une courbe non dégénérée~$C\subset\gm^m$ et un point $P_0\in\gm^{n-m}$
tels que  $X=C\times\{P_0\}$. On peut même supposer, et on le fait,
que le stabilisateur de~$C$ dans~$\gm^m$ est réduit à~$\{1\}$.

Puisque $X$ n'est contenue dans aucun sous-groupe algébrique 
strict de~$\gm^n$,
il en est de même du point~$P_0$ dans~$\gm^{n-m}$. En outre,
comme on s'intéresse à l'intersection de~$X$ avec des sous-groupes
de codimension~$2$ de~$\gm^n$, on peut supposer que $m\geq 2$,
cette intersection étant vide sinon.
Soit $\phi\colon X\times X\ra\gm^n$ le morphisme tel que 
$\phi(P,Q)=P\cdot Q^{-1}$; soit $\psi\colon C\times C\ra\gm^m$
le morphisme analogue. 
Ainsi, on a $\phi((P,P_0),(Q,Q_0))=(\psi(P,Q),1,\dots,1)$
pour tout couple~$(P,Q)\in C\times C$.
Soit $S$ l'adhérence de $\phi(C\times C)$ pour la topologie de Zariski 
dans~$\gm^m$, c'est une surface contenue dans~$\gm^m\times\{(1,\dots,1)\}$
et elle n'est pas dégénérée dans ce sous-groupe,
car, sinon, $C$ serait dégénérée dans~$\gm^m$.

Soit $\bar X$ l'adhérence de~$X$ dans~$\P^n$,
soit $W$ celle de~$S$.
En résolvant les indéterminées de l'application rationnelle~$\phi$
de~$\bar X\times\bar X$ dans~$W$,
on obtient une surface projective~$V$ contenant~$X\times X$ comme ouvert dense,
munie d'un morphisme birationnel et propre~$\pi$ vers~$\bar X\times\bar X$
et d'un morphisme génériquement fini vers~$W$ prolongeant~$\phi$;
notons-le encore~$\phi$.
 Par suite, le fibré en droites $\phi^*\mathscr O(1)$ sur~$W$
est big et nef (gros et numériquement effectif);
décomposons-le sous la forme $\mathscr L(E)$, où $\mathscr L$
est un $\Q$-diviseur ample et $E$ est un $\Q$-diviseur effectif.
Cela entraîne, sur l'ouvert dense $U=V\setminus \abs E$,
une inégalité de hauteurs
\[ h_{\mathscr O(1)}( \phi(P))  = h_{\phi^*\mathscr O(1)}(P)
\geq h_{\mathscr L}(P)  + c_1 \]
pour tout point $P\in V(\bQ)$ n'appartenant pas à~$\abs E$.
En tirant par~$\pi$ le fibré ample $\mathscr O(1)\boxtimes\mathscr O(1)$
sur~$\bar X\times\bar X$  on obtient un fibré en droites
big et nef sur~$V$ et une inégalité de hauteurs
\[ h_{\mathscr L}(P) \gg h(P_1)+h(P_2)-1, \]
où $(P_1,P_2)=\pi(P)$.
Le fermé complémentaire $V\setminus U$
est une réunion finie de courbes et de points. Si la
restriction de~$\phi$ à une telle courbe~$Y\subset V\setminus U$
est un morphisme fini,
on a encore une inégalité similaire. 
Sinon, $\phi(Y)$ est réduit à un point~$R$ de~$\gm^n$; 
si de plus $Y$ rencontre~$X\times X$, 
on constate que $R$ appartient au stabilisateur
de~$X$ dans~$\gm^m$, donc $R=1$.

On en déduit ainsi une inégalité 
\begin{equation}\label{eq.hauteurs}
 h( \phi(P,Q)) \gg h(P)+h(Q)- 1, 
\end{equation}
valable pour tout couple~$(P,Q)$ de~$X(\bQ)\times X(\bQ)$
tel que $P\neq Q$.

Cette analyse entraîne, pour un point~$(P,Q)$ de~$X\times X$,
l'\og alternative\fg suivante:
\begin{enumerate}
\item $P=Q$;
\item $P\neq Q$ et $\phi(P,Q)\in S^\oa$;
\item $P\neq Q$ et $\phi(P,Q)$ appartient à une courbe
atypique maximale de~$S$.
\end{enumerate}
Précisément, on va prouver la finitude de l'ensemble
$X\cap G^{[2]}$ en considérant un point~$P$ de cet ensemble
et en choisissant~$Q$ de la forme $\sigma(P)$,
où $\sigma \in \Gal(\bQ/K)$. 

Il s'agit de conclure
à la finitude dans chacun des trois cas, qu'on discute maintenant
un par un.

Quitte à remplacer~$K$ par une extension finie, on suppose
que toutes les courbes atypiques maximales de~$S$
sont de la forme $Y=S\cap R\cdot H$, où $H$ est
un sous-tore minimal de~$\gm^m$ de codimension~$\geq 2$
et $R$ un point $K$-rationnel de~$\gm^m(K)$.
On suppose aussi que le point~$P_0$ est $K$-rationnel.
Rappelons enfin que l'on raisonne par récurrence
et que l'on suppose le théorème~\ref{theo.maurin'}
vrai pour une courbe d'un tore de dimension~$<n$
qui n'est pas contenue dans un sous-tore strict.

\begin{lemm}
L'ensemble $X(K)\cap G^{[2]}$ est fini.
\end{lemm}
\begin{proof}
En analysant les équations des $n(n-1)/2$ projections évidentes
de~$X\subset\gm^n$ sur~$\gm^2$, on prouve qu'il
existe une famille finie~$\Delta\subset\Z^n$
et une famille finie de places~$\Sigma$ de~$K$ telles que,
pour tout point~$P=(\xi_1,\dots,\xi_n)\in X(K)$ et toute 
valuation discrète~$v\not\in\Sigma$, 
le vecteur $v(P)=(v(\xi_1),\dots,v(\xi_n))$ soit
proportionnel à l'un des éléments de~$\delta$.

Les points~$P\in X(K)$ tels que $v(P)=0$ pour tout~$v\not\in\Sigma$
sont des points $\Sigma$-entiers de~$X$. Leur finitude
est garantie par le théorème de \textsc{Liardet}~\cite{liardet:1975}
sur la conjecture de Mordell--Lang pour les courbes dans~$\gm^n$.

Si $P$ est un point de~$X(K)$, toute valuation~$v\not\in\Sigma$
telle que $v(P)\neq 0$ fournit une contrainte sur
les sous-groupes possibles contenant~$P$. Supposons
en effet $v(P)$ parallèle à un élément~$\delta$ de~$\Delta$.
Quitte à changer
les coordonnées sur~$\gm^n$, on peut supposer $\delta=(0,\dots,0,1)$,
et tout sous-groupe de~$\gm^n$ contenant~$P$ est contenu
dans~$\gm^{n-1}\times \{1\}$.
Si $P\in G^{[2]}$, il en est alors de même de
la projection de~$P$ dans~$\gm^{n-1}$ et l'hypothèse
de récurrence conclut à la finitude voulue.
\end{proof}

\begin{lemm}
L'ensemble des points~$P\in X(\bQ)\cap G^{[2]}$ pour lesquels
il existe $\sigma\in \Gal(\bQ/K)$ tel que $\phi(P,\sigma(P))\in S^\oa$
est fini.
\end{lemm}
\begin{proof}
D'après le théorème~\ref{theo.habegger}, il existe
un nombre réel~$B$ tel que $h(\phi(P,\sigma(P)))\leq B$.
Comme $h(\sigma(P))=h(P)$,
l'inégalité de hauteurs~\eqref{eq.hauteurs}
implique alors que $h(P)$ est majoré sur l'ensemble considéré
par le lemme.
On conclut alors par la proposition~\ref{prop.fini}.
\end{proof}

\begin{lemm}
Soit $Y\subset S$ une courbe atypique maximale.
L'ensemble des points~$P\in X(\bQ)\cap G^{[2]}$
pour lesquels il existe $\sigma\in \Gal(\bQ/K)$
tel que $P\neq \sigma(P)$ et  $\phi(P,\sigma(P))\in Y$ 
est fini.
\end{lemm}
\begin{proof}
Par hypothèse, il existe un sous-tore~$H$
de $\gm^m$ et un point rationnel~$R\in\gm^m(K)$
tel que $Y$ soit une composante de~$S\cap R\cdot H$.
Appliquons~$\sigma,\sigma^2,\dots$ à  la relation $P\cdot \sigma(P)^{-1}\in R\cdot H$ et faisons-en le produit. 
Si $e$ est l'ordre de~$\sigma$ dans~$\Gal(\bQ/K)$,
on obtient que $R^e\in H$. Dans des coordonnées de~$\gm^m$
où $H=\gm^k\times\{(1,\dots,1)\}$, on voit que
les coordonnées d'indices~$k+1$ à~$m$ de~$R$ sont des racines
de l'unité contenues dans~$K$ ; il en résulte
que $R\cdot H=R_1\cdot H$, où $R_1\in\gm^m(K)$ est d'ordre~$e$.

Soit $A$ (\resp $B$) le sous-groupe de~$\Z^n$ 
formé des vecteurs~$(a_1,\dots,a_n)$
tels que le caractère $x\mapsto x_1^{a_1}\dots x_n^{a_n}$
s'annule sur~$P$ (\resp sur~$\phi(P,\sigma(P))$).
Sauf si $\phi(P,\sigma(P))$ appartient à un ensemble fini de points 
exceptionnels de~$Y$, 
le sous-groupe~$B$ annule~$H$.
Supposons pour l'instant que ce soit le cas.
Alors, $B$ contient~$A+\{(0,\dots,0)\}\times\Z^{n-m}$.
En passant au quotient par~$H$, on obtient une courbe~$X_H$ de~$\gm^n/H$
qui n'est contenue dans aucun sous-groupe algébrique strict;
de plus, les relations de dépendance satisfaites par~$P$
en fournissent pour l'image~$P_H$ de~$P$ ; on en déduit
la finitude requise par récurrence.

Il reste à traiter le cas où $\phi(P,\sigma(P))$
est un point exceptionnel de~$Y$. Dans ce cas, quitte
à étendre le corps~$K$, on peut supposer que $Q=\phi(P,\sigma(P))\in \gm^n(K)$.
Si $e$ est l'ordre de~$\sigma$ dans~$\Gal(\bQ/K)$,
on a de nouveau $Q^e=1$ ; quitte à remplacer~$P$ et~$X$
par leur image par l'élévation à la puissance~$e$,
on se ramène au cas où $e=1$, ce qui contredit le fait
que $\sigma(P)\neq P$.
\end{proof}

Cela conclut la preuve du théorème~\ref{theo.maurin'}.

\subsection{Inégalité de Vojta et sous-groupes de rang fini}

Les conjectures~\ref{conj.pink} et~\ref{conj.pink-ml}
généralisent les questions de Manin--Mumford et Mordell--Lang.
On peut en fait les considérer comme
des variantes uniformes de ces questions modulo tous les quotients
de la variété semi-abélienne~$G$ qui sont de dimension~$>\dim(X)$.
La stratégie de \textsc{Rémond} consiste à appliquer
une version uniforme 
des méthodes développées par \textsc{Vojta}, \textsc{Faltings}
et \textsc{McQuillan}, ainsi que de la version~\cite{bombieri1990}
qu'en a donnée \textsc{Bombieri},  pour établir la conjecture de Mordell--Lang
dans les variétés semi-abéliennes.
Il l'a développée dans plusieurs articles 
(\cite{remond2000c,remond2001,remond:2002,remond:2005}), 
l'application à la question du présent rapport fait
l'objet des articles~\cite{remond-viada:2003,remond:2007,remond:2009}
(le premier, en collaboration avec E.~\textsc{Viada})
et est résumée dans l'article de survol~\cite{remond:2009b}
et les lignes qui suivent n'en sont qu'une rapide synthèse.
Pour simplifier la discussion, nous supposons ici que $G$ est une variété
abélienne (\textsc{Maurin} a traité le cas des tores dans~\cite{maurin:2011},
le cas des variétés semi-abéliennes est encore ouvert).

Soit donc $X$ une sous-variété (fermée, irréductible) d'une variété
abélienne~$G$ ; posons $m=\dim(X)$.
Si $X$ est une courbe de genre~$\geq 2$,
l'inégalité de Vojta compare de façon
uniforme la hauteur (de Néron--Tate, \emph{i.e.}, normalisée et symétrique)
d'un couple~$(x,y)$ de~$X\times X$ à celle du point~$ax-y$:
elle affirme qu'il existe des nombres réels~$c_1,c_2,c_3$
tels que
\[ h(ax-y)\geq c_1 (a^2 h(x)+h(y)) \]
pour tout couple~$(x,y)$ de points de~$X^2(\bQ)$ et
tout entier naturel~$a$ tels que $a\geq c_2$, $h(x)\geq c_3$, $a^2h(y)\geq c_3$.
Si l'on choisit $a^2$ proche de~$h(y)/h(x)$, on obtient une minoration
de l'angle formé par les droites~$\R x$ et~$\R y$ dans l'espace
vectoriel réel $G(\bQ)\otimes\R$.
Lorsque $X$ et~$G$ sont définies sur un corps de nombres~$K$,
et $x$, $y$ appartiennent à~$X(K)$, le théorème de Mordell--Weil
entraîne alors que l'ensemble des points de~$X(K)$ de hauteur~$\geq c_3$
est fini; il en est par suite de même de~$X(K)$ lui-même
et l'on a prouvé, suivant~\textsc{Vojta}, la conjecture de Mordell.

Revenons au cas général et
soit $Z_0$ la réunion des translatés de sous-variétés 
abéliennes de~$G$ qui sont contenues dans~$X$ ; c'est l'ensemble exceptionnel
pour le problème de Mordell--Lang.
Soit $a=(a_0,\dots,a_m)\in\Z^{m+1}$; si $Z_0\neq X$, le morphisme
\[ \beta_a\colon  X^{m+1} \ra G^{m}  ,
\quad (x_0,\dots,x_m)\mapsto (a_1x_1-a_0x_0,\dots,a_{m}x_{m}-a_{m-1} x_{m-1}) \]
est génériquement fini et l'inégalité de Vojta affirme que
\[ \sum_{i=1}^{m} h(a_{i} x_{i} -a_{i-1}x_{i-1})  
   \gg \sum_{i=0}^m a_i^2 h(x_i) \]
pour tout $x\in (X\setminus Z_0)^{m+1}(\bQ)$ et tout $a\in\Z^{m+1}$
tels que pour tout~$i$, $a_{i+1}\ll a_i$ et $a_i^2h(x_i)\gg a_0^2$.

Soit maintenant $\Gamma$ un sous-groupe de rang fini de~$G(\bQ)$;
nous supposerons toujours que $\Gamma$ est \emph{saturé},
c'est-à-dire que tout endomorphisme surjectif de~$G$
induit  un endomorphisme surjectif de~$\Gamma$.
Dans~\cite{remond:2009}, \textsc{Rémond} introduit  
plusieurs candidats pour un lieu exceptionnel.
La généralisation la plus évidente de l'ensemble~$X\setminus X^\ta$ est définie
ainsi: il s'agit de l'ensemble~$ Z_{X,\Gamma} $ 
des points~$x\in X$ pour lesquels il existe une sous-variété abélienne~$H$
de~$G$ et un point~$\gamma$ de~$\Gamma$
tels que 
$\dim_x (X\cap (\gamma+H))>\max(0,\dim(X)-\codim(H))$.

En vue de généraliser la définition de l'ensemble~$X^\oa$
(son complémentaire dans~$X$, plutôt),
\textsc{Rémond} définit trois séries d'ensembles exceptionnels,
indexées par un entier~$r\in\{0,\dots,\dim(G)\}$:
\begin{itemize}
\item
L'ensemble~$Z_{X,\ano}^{(r)}=X\setminus X^{\oa,[r]}$ 
formé des tels points~$x$
pour lesquels il existe une sous-variété abélienne~$H$ 
de~$G$ telle que\footnote
{Les lettres~$\ano$ signifient \emph{anomalous},
\textsc{Rémond} utilise la notation $\mathrm{an}$.}
\[ \dim_x(X\cap(x+H))>\max(0,r-1-\codim(H)). \]
\item 
L'ensemble~$Z_{X,\Q}^{(r)}$ réunion des sous-variétés fermées
irréductibles~$Y$ de~$X$ telles qu'il existe des diviseurs 
effectifs~$\mathscr L_1$ et~$\mathscr L_2$ sur~$G$
et un entier~$a$ 
tels que 
\begin{gather*}
 \deg(c_1(\mathscr L_1)^a c_1(\mathscr L_2)^{\dim(Y)-a} \cap [Y])=0 \\
   \rg(\mathscr L_1)-1\leq a\leq\dim(Y) \\
  \rg(\mathscr L_1+\mathscr L_2)\geq r.
\end{gather*}
(Le rang d'un fibré en droites~$\mathscr L$
est le plus grand entier naturel~$r$ tel  que $c_1(\mathscr L)^r$
ne soit pas numériquement trivial.)
\item
L'ensemble $Z_{X}^{(r)}$ défini de façon analogue à~$Z_{X,\Q}^{(r)}$
mais en considérant des diviseurs à coefficients réels.
\end{itemize}
D'après le théorème~1.4 de~\cite{remond:2009}, 
ce sont des parties fermées de~$X$ 
qui vérifient
\[ Z_{X,\Q}^{(r)} \subset Z_X^{(r)} \subset Z_{X,\ano}^{(r)}. \]
Observons aussi que lorsque~$\Gamma$ parcourt
l'ensemble des sous-groupes de rang fini de~$G(\bQ)$,
la réunion des $Z_{X,\Gamma}$ est égale à~$Z_{X,\ano}^{(1+\dim(X))}$.

\medskip

Grâce au théorème de complète réductibilité de Poincaré, toute
sous-variété abélienne~$H$ de~$G$ est la composante
neutre du noyau d'un endomorphisme~$\phi$
de~$G$.
Plutôt que de travailler dans~$G/H$, on peut appliquer~$\phi$.
On peut en outre se restreindre à un ensemble~$\Phi$ d'endomorphismes
qui sont presque des projecteurs, au sens où il existe un entier
naturel~$a$ tel que $\phi\circ\phi= a\phi$ et $\norm{\phi}\ll a$
(on a fixé une norme sur l'espace vectoriel réel~$\End(G)_\R$),
et dont les images forment un ensemble fini de sous-variétés
abéliennes de~$G$.

 
L'inégalité de Vojta uniforme que démontre \textsc{Rémond} est la suivante:
\begin{theo}[\cite{remond:2007}, proposition~5.1]
\label{theo.vojta-uniforme}
Avec ces notations, 
il existe des nombres réels~$c_1,c_2,c_3>0$ tels que l'on ait
\[ \sum_{i=1}^m h(\phi(a_{i}x_{i}-a_{i-1}x_{i-1}))
\geq c_1 \norm{\phi}^2 \sum_{i=0}^m a_i^2 h(x_i) \]
pour tout $x\in (X\setminus Z_{X}^{(r)})^{m+1}$, tout $a\in\N^{m+1}$
et tout $\phi\in \Phi$ dont l'image est de dimension~$\geq r$
tels que $a_{i+1}\leq c_2 a_i$ et $h(x_i)\geq c_3$ pour tout~$i$.
\end{theo}

Ce théorème
affirme ainsi l'existence d'une inégalité de hauteurs
pour des points en dehors du lieu exceptionnel~$Z_X^{(r)}$.
L'ensemble~$Z_{X,\ano}^{(r)}$, de définition peut-être plus naturelle,  
est parfois plus gros.  En outre, 
la proposition~1.3 de~\cite{remond:2007} (déjà mentionnée)
montre qu'il est nécessaire d'exclure les points de lieu~$Z_{X,\Q}^{(r)}$;
il est possible que ce soit aussi suffisant.

La preuve du théorème~\ref{theo.vojta-uniforme}
consiste en la vérification, délicate, des
hypothèses de l'\og inégalité de Vojta généralisée\fg 
que \textsc{Rémond} avait établie dans~\cite{remond:2005}. 
Il faut notamment établir
des minorations uniformes de degrés d'intersection, parmi
lesquelles:
\[ \deg( c_1(\beta_a^* (\phi,\dots,\phi)^*  \mathscr L)^{m(m+1)} \cap [X^{m+1}])\geq 
c_4 (a_0\dots a_m)^{2m} \norm{\phi}^{2m(m+1)} , \]
où $c_4$ est un nombre réel strictement positif.
L'existence de~$c_4$,
lorsque $a$ et~$\phi$ sont fixés, 
équivaut à ce que l'homomorphisme $(\phi,\dots,\phi)\circ \beta_a$
soit génériquement fini. \textsc{Rémond} établit cette
minoration uniforme en combinant des propriétés d'homogénéité
du membre de gauche (qui permettent de supposer que $a=(1,\dots,1)$)
et la possibilité d'étendre de façon continue
ce degré d'intersection au cas où $\phi$ appartient à~$\End(G)_\R$:
la définition de l'ensemble exceptionnel~$Z_X^{(r)}$ assure
que ce degré soit toujours strictement positif,
de même que toutes les variantes requises par le théorème
de~\cite{remond:2005}.

Pour toute partie~$\Sigma$ de~$G(\bQ)$ et tout nombre réel~$\eps>0$,
notons $\mathscr B(\Sigma,\eps)$ l'ensemble\footnote
{La lettre~$\mathscr B$ est l'initiale du mot boule.}
des points de~$G(\bQ)$
de la forme $x+y$, où $x\in\Sigma$ et $\hat h(y)\leq\eps$.

\begin{coro}[\cite{remond:2007}, Théorème~1.2]
Il existe un nombre réel~$\eps>0$ tel que l'ensemble 
$(X(\bQ)\setminus Z_{X}^{(r)})\cap \mathscr B(\Gamma+G^{[r]},\eps)$
soit de hauteur bornée.
\end{coro}

\begin{rema}
Il convient de noter que dans ce résultat,
il n'y a pas le décalage d'exposants que l'on trouvait
dans le théorème~\ref{theo.habegger}. C'est en
quelque sorte le prix à payer
pour permettre de traiter des sous-groupes de rang fini.
En effet, d'après la proposition~1.3 de~\cite{remond:2007},
si $Z$ est un fermé de~$X$
tel que $(X(\bQ)\setminus Z)\cap \mathscr B(\Gamma+G^{[r]},\eps)$
est de hauteur bornée pour tout sous-groupe~$\Gamma$
de rang fini de~$G(\bQ)$, alors $Z$ contient~$Z_{X,\Q}^{(r)}$.
\end{rema}

\begin{proof}
Pour simplifier, on ne traite que la variante du corollaire
sans~$\eps$ et on démontre que l'ensemble 
$(X(\bQ)\setminus Z_X^{(r)})\cap (\Gamma+G^{[r]})$
est de hauteur bornée.
Considérons, par l'absurde, une suite $(x_n)$ dans cet ensemble dont
la hauteur tend vers l'infini. On écrit $x_n=\gamma_n+P_n+Q_n$,
où $\gamma_n\in\Gamma$ et $P_n$ appartient à une sous-variété
abélienne de codimension~$\geq\dim(X)$.
On fixe un endomorphisme~$\phi_n\in\Phi$ tel que $\phi_n(P_n)=0$;
on peut en outre supposer que $\max(h(\gamma_n),h(P_n)) \ll h(x_n)$.
Enfin, comme $\End(G)_\R$ et $\Gamma\otimes_\Z\R$ sont des
$\R$-espaces vectoriels de dimension finie,
et quitte à considérer des sous-suites convenables de la suite~$(x_n)$,
on peut supposer que 
$\phi_n/\norm{\phi_n}$ et
$\gamma_n/h(\gamma_n)$ 
sont assez proches d'un même élément de $\End(G)_\R$,
\resp de~$\Gamma_\R$.
Au prix d'une éventuelle extraction supplémentaire,
il est alors possible de construire des entiers $a_0,\dots,a_m$
de sorte à contredire
le théorème~\ref{theo.vojta-uniforme} (avec $\phi=\phi_0$).
\end{proof}

Lorsque $X$ n'est pas géométriquement dégénérée, 
$X^\oa\subset X\setminus Z_X^{(1+\dim(X))}$ n'est pas vide 
et le corollaire entraîne donc que
$X^\oa(\bQ)\cap (\Gamma+G^{[\dim(X)]})$ est de hauteur bornée.
Si, de plus, $G$ est à multiplications complexes, 
la proposition~\ref{prop.fini-abelien} entraîne que cet ensemble est fini,
démontrant du même coup le théorème~\ref{theo.nd.gamma}.

L'analogue abélien du théorème~\ref{theo.maurin'} s'ensuit
facilement:
\begin{coro}[\cite{remond:2007}, Corollaire~1.6]
\label{coro.remond}
Soit $X$ une courbe irréductible d'une variété abélienne
à multiplications complexes~$G$. Si $X$ n'est contenue
dans aucun sous-groupe  algébrique strict de~$G$,
alors $X(\bQ)\cap G^{[2]}$ est fini.
\end{coro}
\begin{proof}
Soit $H$ la plus petite sous-variété abélienne de~$G$ contenant
un translaté de~$X$ et soit $g$ un point de~$X$
tel que $X\subset g+H$. Alors, la courbe $Y=X-g$ est non dégénérée
dans~$H$, tandis que $g$ engendre $G/H$. Soit $\Gamma$
le plus petit sous-groupe saturé de~$H(\bQ)$  contenant
l'image de~$g$ par tout homomorphisme de~$G$ dans~$H$;
c'est un sous-groupe de rang fini de~$H(\bQ)$.
On constate que pour tout point~$x$ de~$X(\bQ)\cap G^{[2]}$,
$x-g$ appartient à~$Y(\bQ)\cap (\Gamma+ H^{[2]})$.
Appliqué à~$Y$ et au groupe~$\Gamma$,
le théorème~\ref{theo.nd.gamma} entraîne donc
que $X(\bQ)\cap G^{[2]}$ n'est pas dense dans~$X$.
Comme $X$ est une courbe, il est donc fini.
\end{proof}

\begin{rema}
Lorsqu'on oublie la partie~$G^{[r]}$, les énoncés précédents
fournissent un cas particulier 
de l'énoncé suivant, combinaison des conjectures de Mordell--Lang
et Bogomolov: \emph{Si $X$ n'est pas un translaté
d'un sous-groupe algébrique de~$G$, 
l'intersection de $X$ et de~$\mathscr B(\Gamma,\eps)$
n'est pas dense dans~$X$ pour la topologie de Zariski.}
Ce résultat a été établi par B.~\textsc{Poonen}~\cite{poonen99}
(pour les produits de variétés abéliennes et de tores)
et par \textsc{Rémond}~\cite{remond:2003} en général.
\end{rema}

\subsection{Conjecture de Bogomolov et finitude}

\textsc{Habegger}~\cite{habegger:2009b}, \textsc{Maurin}~\cite{maurin:2011}
et \textsc{Viada}~\cite{viada:2008,viada:2010}
ont montré comment
utiliser les versions effectives
de la conjecture de Bogomolov 
pour déduire du théorème~\ref{theo.habegger} (ou de ses variantes)
des énoncés de finitude.
Ces deux premiers auteurs traitaient le cas des tores, 
la dernière des variétés abéliennes, cas sur lequel nous allons
nous concentrer maintenant.
Pour un énoncé similaire au théorème~\ref{theo.viada}
dans le cas des tores, voir le théorème~11.6 de~\cite{maurin:2011}.

Dans tout ce paragraphe, on considère donc une variété
abélienne~$G$ définie sur~$\bQ$, un fibré en droites ample~$\mathscr L$ sur~$G$
et la hauteur canonique~$\hat h_{\mathscr L}$ associée.

Avec ces notations, \textsc{Viada} démontre le théorème suivant:
\begin{theo}[\cite{viada:2010}, Theorem~1.6]
\label{theo.viada}
On suppose vérifiée la conjecture~\ref{conj.bogomolov}.
Soit $X$ une sous-variété (fermée, irréductible) de~$G$, distincte de~$G$.

\begin{enumerate}
\item Supposons que $X$ ne soit contenue dans 
aucun sous-groupe algébrique strict de~$G$.
Pour tout nombre réel~$B$, il existe un nombre réel~$\eps>0$
tel que l'ensemble des points de~$X\cap \mathscr B(G^{[1+\dim(X)]},\eps)$
dont la hauteur est~$\leq B$ ne soit pas dense dans~$X$
pour la topologie de Zariski.
\item Supposons que $X$  ne soit contenue dans aucun translaté
de sous-variété abélienne de~$G$.
Soit $\Gamma$ un sous-groupe de rang fini de~$G$.
Pour tout nombre réel~$B$, il existe un nombre réel~$\eps>0$
tel que 
l'ensemble des points de~$X\cap \mathscr B(\Gamma\cdot G^{[1+\dim(X)]},\eps)$
dont la hauteur est~$\leq B$ ne soit pas dense dans~$X$
pour la topologie de Zariski.
\end{enumerate}
\end{theo}

En fait, la preuve ne requiert la conjecture~\ref{conj.bogomolov}
que sous une forme affaiblie, et
pour un ensemble fini de sous-quotients de puissances de~$G$,
ce qui permet, dans cet énoncé, 
de choisir la constante~$c$ indépendamment de~$G$.

\textsc{Viada} commence par établir l'équivalence
des deux assertions du théorème; la démonstration
se concentre alors sur la seconde.

De manière analogue à ce qui a été fait dans
le paragraphe sur l'inégalité de Vojta,
on écrit une sous-variété abélienne
de codimension au moins~$1+\dim(X)$
comme la composante neutre 
du noyau d'un homomorphisme surjectif. 
Pour simplifier les notations, supposons ici que $G$ est une puissance~$A^g$
d'une variété abélienne simple~$A$ de dimension~$a$.
Notons $E=\End_{\bQ}(A)$; c'est un sous-anneau du corps~$E\otimes\Q$.
Il suffit de considérer des homomorphismes surjectifs 
de but~$A^r$, où $r\in\{0,\dots,g\}$
est un entier tel que $ar\geq 1+\dim(X)$.
Ces homomorphismes sont donnés par une matrice~$r\times g$ à coefficients
dans~$E$.
Un argument de réduction de Gauß
couplé avec un énoncé d'approximation diophantienne
(lemme de Dirichlet) ramène à considérer un ensemble
fini~$\Phi$ d'homomorphismes $\phi\colon A^{g+s}\ra A^r$
dont la matrice est de la forme 
$\begin{pmatrix} m \mathrm I_r & L & L'\end{pmatrix} $
et est de norme $\ll \abs m$.

Il reste à prouver que pour tout $\phi\in\Phi$ et
tout point~$\gamma$ de~$A^s$, l'ensemble des points
de $A^{g+s}$ de la forme $(x,\gamma)$, 
où $x\in X$ est de hauteur~$\leq B$, 
où  $(x,\gamma)\in\mathscr B(\ker(\phi),{\eps/\norm{\phi}^2})$
n'est pas dense dans~$X\times\{\gamma\}$
pour la topologie de Zariski.
En revanche, le processus d'approximation modifie~$\eps$
et on n'a pas de contrôle a priori de~$\norm\phi$.
Il suffit cependant de trouver un nombre réel~$\eps$
qui entraîne la non-densité pour tout homomorphisme~$\phi$
comme ci-dessus.
C'est là qu'intervient la forme effective conjecturale
de la conjecture de Bogomolov.

Lorsque $\norm\phi$ n'est pas trop grande,
on obtient directement la non-densité voulue 
en considérant l'image par~$\phi$
de~$X\times\{\gamma\}$ dans~$A^s$
et en notant qu'elle est de codimension au moins~$1$
et de degré~$\ll\norm{\phi}^{2\dim(X)}$.
Les images par~$\phi$
dans~$\phi(X\times\{\gamma\})$
des points $p=(x,\gamma)$  considérés
sont en effet de hauteur $\ll\eps\norm{\phi}^2$;
leur densité entraîne que $\eps\norm{\phi}^{2\dim(X)+2}\gg 1$.

Dans l'autre cas, on considère l'isogénie~$\psi$
de~$A^g$ de matrice~$ \begin{pmatrix} m\mathrm I_r & L \\ 0 & \mathrm I_{g-r}\end{pmatrix}$
et une composante irréductible~$Y$ de~$[m]^{-1}\psi(X)$
dans~$A^g$; on démontre que son degré est majoré par un multiple
de $\norm{\phi}^{2a(g-r)}$.
Si $p=(x,\gamma)$ vérifie les relations envisagées,
on constate que $\hat h(\psi(x))\ll\max(B,\eps\norm{\phi}^2)$.
Sous l'hypothèse qu'ils forment un ensemble dense dans~$X\times\{\gamma\}$,
la version effective de la conjecture de Bogomolov entraîne
que $\max(B,\eps\norm{\phi}^2)\ll \norm{\phi}^{2/(ag-\dim(X)}$.

On vérifie que si $\eps$ est assez petit (indépendamment de~$\phi$),
ces deux inégalités sont toutes deux fausses, ce qui conclut la démonstration.

\medskip

Sous l'hypothèse que $G$ a une densité positive de réductions ordinaires,
le théorème~\ref{theo.galateau} de \textsc{Galateau} fournit 
une version effective de la conjecture de Bogomolov
à peine plus faible que celle conjecturée.
En outre, cette hypothèse est vérifiée pour l'ensemble fini de quotients
de~$G$ considérés dans la preuve.
En reprenant les calculs ci-dessus, on constate
que cela suffit pour assurer la conclusion du théorème~\ref{theo.viada}.

Compte tenu du théorème~\ref{theo.habegger-epsilon},
il en résulte le théorème:
\begin{theo}
Soit $G$ une variété abélienne banale et soit $X$
une sous-variété (fermée, irréductible) non dégénérée dans~$G$.
Soit $\Gamma$ un sous-groupe de rang fini de~$G(\bQ)$.
Il existe un nombre réel~$\eps>0$
tel que $X(\bQ)\cap \mathscr B(G^{[1+\dim(X)]}+\Gamma,\eps)$
ne soit pas dense dans~$X$ pour la topologie de Zariski.
\end{theo}

En raisonnant comme dans la preuve du corollaire~\ref{coro.remond},
on en déduit un résultat  pour les courbes:
\begin{coro}
Soit $G$ une variété abélienne banale et soit $X$ une courbe
(fermée, irréductible) dans~$G$ qui n'est contenue dans
aucun sous-groupe algébrique strict de~$G$.
Il existe un nombre réel~$\eps>0$
tel que $X(\bQ)\cap \mathscr B(G^{[1+\dim(X)]},\eps)$
soit fini.
\end{coro}

\subsection{Familles de variétés semi-abéliennes}

Comme je l'ai évoqué dans l'introduction,
\textsc{Pink}  a proposé une généralisation commune
des conjectures d'André--Oort, Manin--Mumford, Mordell--Lang
dans le cadre des variétés de Shimura mixtes.
Une variante de cette dernière conjecture, d'énoncé plus élémentaire,
concerne les familles de variétés semi-abéliennes.

Soit $S$ une variété algébrique complexe et soit $p\colon G\ra S$
un schéma semi-abélien sur~$S$, c'est-à-dire une
famille de variétés semi-abéliennes paramétrée par~$S$.
Pour tout entier~$r$, on note $G^{[r]}$ la réunion, pour $s\in S$,
des sous-ensembles $G_s^{[r]}$ de~$G_s$: un point $g$ de~$G$,
d'image $p(g)\in S$,
appartient à~$G^{[r]}$ si et seulement s'il appartient à 
un sous-groupe algébrique de codimension~$\geq r$
de sa fibre~$G_{p(g)}$.

Dans ces conditions, \textsc{Pink} conjecturait l'énoncé suivant:
\begin{conj}[\cite{pink:2005}, Conjecture~6.2]\label{conj.famille}
Soit $p\colon G\ra S$ un schéma semi-abélien
et soit $X\subset G$ uns sous-variété fermée irréductible.
Si $X$ n'est pas contenue dans un sous-schéma en groupes
strict de~$G$, alors $X\cap G^{[1+\dim(X)]}$ n'est pas
dense dans~$X$ pour la topologie de Zariski.
\end{conj}

Donnons un exemple: prenons pour $S$ le complémentaire de~$\{0,1,\infty\}$
dans la droite projective et soit $p\colon E\ra S$
la famille de Legendre des courbes elliptiques,
définie par l'équation affine $y^2=x(x-1)(x-s)$ dans~$\mathbf A^2_S$.
Soit alors $G=E\times_S E$, le produit fibré de deux copies de~$E$.
L'ensemble $G^{[2]}$ est l'ensemble des couples $(s,e_1,e_2)$
où $s\in S$ et  $e_1,e_2$ sont des points de torsion de la courbe~$E_s$.
Dans ce cas, \textsc{Masser} et \textsc{Zannier}
démontrent
le théorème suivant (\cite{masser-zannier:2011},
voir aussi~\cite{masser-zannier:2010}):
\begin{theo}
Soit $E\ra S$ une famille non constante de courbes elliptiques
et soit $G$ le schéma abélien $E\times_S E$.
Soit $X$ une courbe fermée irréductible dans $G$.
Alors, $X\cap G^{[2]}$ est contenu dans une réunion 
finie de sous-schémas abéliens stricts de~$G$.
\end{theo}
La preuve de ce théorème, d'une nature assez différente
de celles esquissées dans ce rapport, repose sur l'approche
de la conjecture de Manin--Mumford
découverte par \textsc{Pila} et \textsc{Zannier}~\cite{pila-zannier:2008}
et sur un théorème de \textsc{Pila} et \textsc{Wilkie}
concernant les points rationnels des ensembles définissables
dans une structure o-minimale. 

\medskip

Toutefois, en reprenant une construction 
(due à \textsc{K.~Ribet})
de points spéciaux sur des extensions de schémas abéliens
par un tore, D.~\textsc{Bertrand}~\cite{bertrand:2011}
a observé récemment que
la conjecture~\ref{conj.famille} est fausse.

\begin{theo}[\cite{bertrand:2011}, Theorem~1]\label{theo.bertrand}
Soit $E$ une courbe elliptique à multiplications complexes,
soit $S$ une courbe algébrique complexe
et soit $p\colon G\ra S$ un schéma semi-abélien
qui est extension non constante de $E_S$ par $\mathbf G_{\mathrm m,S}$.
Il existe une section $\beta\colon S\ra G$ dont l'image n'est
contenue dans aucun sous-schéma en groupes strict de~$G$,
mais tel que l'ensemble des points~$s\in S$
tels que $\beta(s)$ soit un point de torsion de~$G_s$
soit infini.
\end{theo}
\begin{proof}
Soit $E'$ la courbe elliptique duale de~$E$ (on a $E\simeq E'$,
mais il est plus pratique de les différencier)
et notons $\mathscr P\ra E\times E'$ la biextension de Poincaré.
L'extension~$G$ est donnée par un morphisme non constant
$\gamma\colon S\ra E'$ ;
l'universalité de la biextension
de Poincaré implique que si l'on note $\gamma_E=\gamma\times \Id_E$,
$\gamma_E^*\mathscr P$ est égal au $\gm$-torseur sur~$E_S$
donné par~$G$.
Il y a ainsi une bijection canonique entre l'ensemble des sections 
$\beta\colon S\ra G$
et celui des couples formé d'un un morphisme~$q$ de~$S$ dans~$E$ et 
d'une trivialisation du
fibré en droites~$(\gamma\times q)^*\mathscr P$ sur~$S$.

Soit $f\colon E'\ra E$ une isogénie  telle que $f'\neq f$
et notons $g$ l'isogénie antisymétrique donnée par $g=f-f'$.
La bidualité  de~$\mathscr P$ entraîne l'existence
d'un isomorphisme canonique
\[ (\gamma,f\circ\gamma)^*\mathscr P
\simeq (\gamma,f'\circ\gamma)^*\mathscr P .\]
Comme $\mathscr P$ est une biextension, on a donc
\[ (\gamma,(f-f')\circ\gamma))^*\mathscr P \simeq \mathscr O_S, \]
d'où l'existence d'une section canonique $\beta\colon S\ra G$
relevant $g\circ\gamma\colon S\ra E$.

Pour tout $s\in S$ tel que $\gamma(s)$ est un point de torsion de~$E$,
on démontre le point~$\beta(s)$ est de torsion.
Toutefois,  comme l'extension~$G$ n'est pas constante,
les seuls sous-schémas en groupes stricts de~$G$
sont ou bien finis sur~$S$, ou bien de la forme $T\gm$,
où $T$ est fini sur~$S$. Comme $g\circ\gamma$ n'est pas constante,
elle n'est pas d'ordre fini et $\beta(S)$ n'est contenue
dans aucun sous-schéma en groupes strict de~$G$.
\end{proof}

Comme l'explique \textsc{Bertrand} dans son article,
ce contre-exemple s'interprète parfaitement dans
le cadre de la conjecture générale de \textsc{Pink}
(\cite{pink:2005}, conjecture~1.3). 
Avec les notations du théorème~\ref{theo.bertrand},
l'image $\beta(S)$ provient d'une sous-variété spéciale d'une
sous-variété de Shimura mixte. Autrement dit,
contrairement au cas absolu,
toutes les sous-variétés spéciales d'un schéma semi-abélien
ne sont pas des translatés de sous-schémas abéliens par
des points de torsion. 

Il est intéressant de remarquer que les points de Ribet analogues
à ceux du théorème précédent, lorsque $S$ est un point,
fournissent des contre-exemples à une généralisation hâtive
des conjectures~\ref{conj.lehmer}  et~\ref{conj.lehmer-relatif}
aux variétés semi-abéliennes générales.

Restreinte au cas des schémas abéliens,
la conjecture~\ref{conj.famille} est une conséquence
de la conjecture générale de Pink et reste ouverte.

\subsection*{Remerciements}
Je voudrais remercier D.~\textsc{Bertrand},
S.~\textsc{David}, A.~\textsc{Galateau}, D.~\textsc{Masser}, 
G.~\textsc{Rémond}, E.~\textsc{Ullmo}
et U.~\textsc{Zannier} pour leur aide 
dans la préparation de ce rapport.

En outre, les dix années qui séparent l'article
fondateur~\cite{bombieri-masser-zannier:1999}
et la preuve des résultats exposés dans ce rapport
ont bien sûr vu 
de nombreux progrès intermédiaires ; je n'ai pas
pu tous les mentionner  et m'en excuse auprès de leurs auteurs.

 \def\bibliofont{
 \itemsep0pt\parsep0pt\parskip0pt}


\bibliographystyle{smfplain}
\bibliography{aclab,acl,newbib2}

\providecommand{\noopsort}[1]{}\providecommand{\url}[1]{\textit{#1}}
\providecommand{\bysame}{\leavevmode ---\ }
\providecommand{\og}{``}
\providecommand{\fg}{''}
\providecommand{\smfandname}{\&}
\providecommand{\smfedsname}{\'eds.}
\providecommand{\smfedname}{\'ed.}
\providecommand{\smfmastersthesisname}{M\'emoire}
\providecommand{\smfphdthesisname}{Th\`ese}
\begin{thebibliography}{10}

\bibitem{abbes1997}
{\scshape A.~Abbes} -- {\og Hauteurs et discr\'etude\fg}, in
  \emph{S{\'e}minaire Bourbaki, 1996/97}, Ast{\'e}risque, no. 245, 1997,
  Exp.~825, p.~141--166.

\bibitem{amoroso-david:1999}
{\scshape F.~Amoroso {\normalfont \smfandname} S.~David} -- {\og Le probl\`eme
  de {L}ehmer en dimension sup\'erieure\fg}, \emph{J. reine angew. math.}
  \textbf{513} (1999), p.~145--179.

\bibitem{amoroso-david2003}
\bysame , {\og Minoration de la hauteur normalis\'ee dans un tore\fg}, \emph{J.
  Inst. Math. Jussieu} \textbf{2} (2003), no.~3, p.~335--381.

\bibitem{amoroso-david2004}
\bysame , {\og Distribution des points de petite hauteur dans les groupes
  multiplicatifs\fg}, \emph{Ann. Sc. Norm. Super. Pisa Cl. Sci. (5)} \textbf{3}
  (2004), no.~2, p.~325--348.

\bibitem{amoroso-viada:2009}
{\scshape F.~Amoroso {\normalfont \smfandname} E.~Viada} -- {\og Small points
  on subvarieties of a torus\fg}, \emph{Duke Math. J.} \textbf{150} (2009),
  no.~3, p.~407--442.

\bibitem{ax:1972}
{\scshape J.~Ax} -- {\og Some topics in differential algebraic geometry. {I}.
  {A}nalytic subgroups of algebraic groups\fg}, \emph{Amer. J. Math.}
  \textbf{94} (1972), p.~1195--1204.

\bibitem{bertrand:2011}
{\scshape D.~Bertrand} -- {\og Special points and {P}oincar{\'e}
  biextensions\fg}, 2011, arXiv.11045178 (with an appendix by {B}as
  {E}dixhoven).

\bibitem{bogomolov80b}
{\scshape F.~A. Bogomolov} -- {\og Points of finite order on abelian
  varieties\fg}, \emph{Izv. Akad. Nauk. SSSR Ser. Mat.} \textbf{44} (1980),
  no.~4, p.~782--804, 973.

\bibitem{bombieri1990}
{\scshape E.~Bombieri} -- {\og The {M}ordell conjecture revisited\fg},
  \emph{Ann. Scuola Norm. Sup. Pisa Cl. Sci. (4)} \textbf{17} (1990), no.~4,
  p.~615--640.

\bibitem{bombieri-gubler2006}
{\scshape E.~Bombieri {\normalfont \smfandname} W.~Gubler} -- \emph{Heights in
  {D}iophantine geometry}, New Mathematical Monographs, vol.~4, Cambridge
  University Press, Cambridge, 2006.

\bibitem{bombieri-habegger-masser-zannier:2010}
{\scshape E.~Bombieri, P.~Habegger, D.~Masser {\normalfont \smfandname}
  U.~Zannier} -- {\og A note on {M}aurin's theorem\fg}, \emph{Atti Accad. Naz.
  Lincei Cl. Sci. Fis. Mat. Natur. Rend. Lincei (9) Mat. Appl.} \textbf{21}
  (2010), no.~3, p.~251--260.

\bibitem{bombieri-masser-zannier:1999}
{\scshape E.~Bombieri, D.~Masser {\normalfont \smfandname} U.~Zannier} -- {\og
  Intersecting a curve with algebraic subgroups of multiplicative groups\fg},
  \emph{Internat. Math. Res. Notices} (1999), no.~20, p.~1119--1140.

\bibitem{bombieri-masser-zannier:2006}
\bysame , {\og Intersecting curves and algebraic subgroups: conjectures and
  more results\fg}, \emph{Trans. Amer. Math. Soc.} \textbf{358} (2006), no.~5,
  p.~2247--2257 (electronic).

\bibitem{bombieri-masser-zannier:2007}
\bysame , {\og Anomalous subvarieties---structure theorems and
  applications\fg}, \emph{Internat. Math. Res. Notices} (2007), no.~19, p.~Art.
  ID rnm057, 33.

\bibitem{bombieri-masser-zannier:2008}
\bysame , {\og Intersecting a plane with algebraic subgroups of multiplicative
  groups\fg}, \emph{Ann. Sc. Norm. Super. Pisa Cl. Sci. (5)} \textbf{7} (2008),
  no.~1, p.~51--80.

\bibitem{bombieri-masser-zannier:2008b}
\bysame , {\og On unlikely intersections of complex varieties with tori\fg},
  \emph{Acta Arith.} \textbf{133} (2008), no.~4, p.~309--323.

\bibitem{bombieri-zannier:1995}
{\scshape E.~Bombieri {\normalfont \smfandname} U.~Zannier} -- {\og Algebraic
  points on subvarieties of {$\mathbf G^n_m$}\fg}, \emph{Internat. Math. Res.
  Notices} (1995), no.~7, p.~333--347.

\bibitem{bost-g-s94}
{\scshape J.-B. Bost, H.~Gillet {\normalfont \smfandname} C.~Soul{\'e}} -- {\og
  Heights of projective varieties and positive {G}reen forms\fg}, \emph{J.
  Amer. Math. Soc.} \textbf{7} (1994), p.~903--1027.

\bibitem{carrizosa:2009}
{\scshape M.~Carrizosa} -- {\og Petits points et multiplication complexe\fg},
  \emph{Internat. Math. Res. Notices} (2009), no.~16, p.~3016--3097.

\bibitem{david-hindry2000}
{\scshape S.~David {\normalfont \smfandname} M.~Hindry} -- {\og Minoration de
  la hauteur de {N}\'eron-{T}ate sur les vari\'et\'es ab\'eliennes de type {C}.
  {M}\fg}, \emph{J. reine angew. Math.} \textbf{529} (2000), p.~1--74.

\bibitem{david-p98}
{\scshape S.~David {\normalfont \smfandname} P.~Philippon} -- {\og Minorations
  des hauteurs normalis{\'e}es des sous-vari{\'e}t{\'e}s de vari{\'e}t{\'e}s
  ab{\'e}liennes\fg}, in \emph{International Conference on Discrete Mathematics
  and Number Theory} (Tiruchirapelli, 1996), Contemp. Math., no. 210, 1998,
  p.~333--364.

\bibitem{david-p99}
\bysame , {\og Minorations des hauteurs normalis\'ees des sous-vari\'et\'es des
  tores\fg}, \emph{Ann. Scuola Norm. Sup. Pisa} \textbf{28} (1999), no.~3,
  p.~489--543.

\bibitem{david-p2000}
\bysame , {\og Sous-vari\'et\'es de torsion des vari\'et\'es
  semi-ab\'eliennes\fg}, \emph{C. R. Acad. Sci. Paris S\'er. I Math.}
  \textbf{331} (2000), no.~8, p.~587--592.

\bibitem{delsinne:2009}
{\scshape E.~Delsinne} -- {\og Le probl\`eme de {L}ehmer relatif en dimension
  sup\'erieure\fg}, \emph{Ann. Sci. \'Ec. Norm. Sup\'er. (4)} \textbf{42}
  (2009), no.~6, p.~981--1028.

\bibitem{dobrowolski1979}
{\scshape E.~Dobrowolski} -- {\og On a question of {L}ehmer and the number of
  irreducible factors of a polynomial\fg}, \emph{Acta Arith.} \textbf{34}
  (1979), no.~4, p.~391--401.

\bibitem{faltings1991}
{\scshape G.~Faltings} -- {\og Diophantine approximation on abelian
  varieties\fg}, \emph{Ann. of Math. (2)} \textbf{133} (1991), no.~3,
  p.~549--576.

\bibitem{galateau:2010}
{\scshape A.~Galateau} -- {\og Le probl\`eme de {B}ogomolov effectif sur les
  vari\'et\'es ab\'eliennes\fg}, \emph{Algebra Number Theory} \textbf{4}
  (2010), no.~5, p.~547--598.

\bibitem{habegger:2008}
{\scshape P.~Habegger} -- {\og Intersecting subvarieties of {${\bf G}^n_m$}
  with algebraic subgroups\fg}, \emph{Math. Ann.} \textbf{342} (2008), no.~2,
  p.~449--466.

\bibitem{habegger:2009b}
\bysame , {\og A {B}ogomolov property for curves modulo algebraic
  subgroups\fg}, \emph{Bull. Soc. Math. France} \textbf{137} (2009), no.~1,
  p.~93--125.

\bibitem{habegger:2009a}
\bysame , {\og Intersecting subvarieties of abelian varieties with algebraic
  subgroups of complementary dimension\fg}, \emph{Invent. Math.} \textbf{176}
  (2009), no.~2, p.~405--447.

\bibitem{habegger:2009d}
\bysame , {\og On the bounded height conjecture\fg}, \emph{Internat. Math. Res.
  Notices} (2009), no.~5, p.~860--886.

\bibitem{hindry1988}
{\scshape M.~Hindry} -- {\og Autour d'une conjecture de {S}erge {L}ang\fg},
  \emph{Invent. Math.} \textbf{94} (1988), no.~3, p.~575--603.

\bibitem{hindry-silverman2000}
{\scshape M.~Hindry {\normalfont \smfandname} J.~H. Silverman} --
  \emph{Diophantine geometry}, Graduate Texts in Mathematics, vol. 201,
  Springer-Verlag, New York, 2000, An introduction.

\bibitem{hrushovski2001}
{\scshape E.~Hrushovski} -- {\og The {M}anin-{M}umford conjecture and the model
  theory of difference fields\fg}, \emph{Ann. Pure Appl. Logic} \textbf{112}
  (2001), no.~1, p.~43--115.

\bibitem{kirby:2009}
{\scshape J.~Kirby} -- {\og The theory of the exponential differential
  equations of semiabelian varieties\fg}, \emph{Selecta Math. (N.S.)}
  \textbf{15} (2009), no.~3, p.~445--486.

\bibitem{lang1983}
{\scshape S.~Lang} -- \emph{Fundamentals of {D}iophantine geometry},
  Springer-Verlag, New York, 1983.

\bibitem{laurent:1983}
{\scshape M.~Laurent} -- {\og Minoration de la hauteur de {N}\'eron-{T}ate\fg},
  in \emph{Seminar on number theory, {P}aris 1981--82 ({P}aris, 1981/1982)},
  Progr. Math., vol.~38, Birkh\"auser Boston, Boston, MA, 1983, p.~137--151.

\bibitem{laurent84}
\bysame , {\og {\'E}quations diophantiennes exponentielles\fg}, \emph{Invent.
  Math.} \textbf{78} (1984), p.~299--327.

\bibitem{liardet:1975}
{\scshape P.~Liardet} -- {\og Sur une conjecture de {S}erge {L}ang\fg}, in
  \emph{Journ\'ees {A}rithm\'etiques de {B}ordeaux ({C}onf., {U}niv.
  {B}ordeaux, {B}ordeaux, 1974)}, Soc. Math. France, Paris, 1975, p.~187--210.
  Ast\'erisque, Nos. 24--25.

\bibitem{masser-zannier:2010}
{\scshape D.~Masser {\normalfont \smfandname} U.~Zannier} -- {\og {Torsion
  anomalous points and families of elliptic curves.}\fg}, \emph{Amer. J. Math.}
  \textbf{132} (2010), no.~6, p.~1677--1691.

\bibitem{masser-zannier:2011}
\bysame , {\og Torsion points on families of squares of elliptic curves\fg},
  \emph{Math. Ann.} (2011).

\bibitem{maurin:2008}
{\scshape G.~Maurin} -- {\og Courbes alg\'ebriques et \'equations
  multiplicatives\fg}, \emph{Math. Ann.} \textbf{341} (2008), no.~4,
  p.~789--824.

\bibitem{maurin:2011}
\bysame , {\og Équations multiplicatives sur les sous-variétés des
  tores\fg}, \emph{Internat. Math. Res. Notices} (2011), 108~p., \`a para\^\i
  tre.

\bibitem{mcquillan1995}
{\scshape M.~McQuillan} -- {\og Division points on semi-abelian varieties\fg},
  \emph{Invent. Math.} \textbf{120} (1995), no.~1, p.~143--159.

\bibitem{ogus82}
{\scshape A.~Ogus} -- {\og {H}odge cycles and crystalline cohomology\fg},
  Lecture Notes in Math., no. 900, ch.~6, p.~357--414, Springer-Verlag, 1982.

\bibitem{pila:2011}
{\scshape J.~Pila} -- {\og O-minimality and the {A}ndr{\'e}-{O}ort conjecture
  for $\mathbf {C}^n$\fg}, \emph{Ann. of Math.} \textbf{173} (2011), no.~3,
  p.~1779--1840.

\bibitem{pila-zannier:2008}
{\scshape J.~Pila {\normalfont \smfandname} U.~Zannier} -- {\og Rational points
  in periodic analytic sets and the {M}anin-{M}umford conjecture\fg},
  \emph{Atti Accad. Naz. Lincei Cl. Sci. Fis. Mat. Natur. Rend. Lincei (9) Mat.
  Appl.} \textbf{19} (2008), no.~2, p.~149--162.

\bibitem{pink:1998}
{\scshape R.~Pink} -- {\og {$l$}-adic algebraic monodromy groups, cocharacters,
  and the {M}umford-{T}ate conjecture\fg}, \emph{J. reine angew. Math.}
  \textbf{495} (1998), p.~187--237.

\bibitem{pink:2005}
\bysame , {\og A combination of the conjectures of {M}ordell-{L}ang and
  {A}ndr{{\'e}}-{O}ort\fg}, in \emph{Geometric Methods in Algebra and Number
  Theory} (F.~Bogomolov {\normalfont \smfandname} Y.~Tschinkel, \smfedsname),
  Progr. Math., vol. 253, Birkh{\"a}user, 2005, p.~251--282.

\bibitem{pink:2005b}
\bysame , {\og A common generalization of the conjectures of
  {A}ndr{{\'e}}-{O}ort, {M}anin-{M}umford, and {M}ordell-{L}ang\fg}, Tech.
  report, ETH, Zürich, 2005,
  \url{http://www.math.ethz.ch/~pink/ftp/AOMMML.pdf}.

\bibitem{pink-roessler:2002}
{\scshape R.~Pink {\normalfont \smfandname} D.~Roessler} -- {\og On
  {H}rushovski's proof of the {M}anin-{M}umford conjecture\fg}, in
  \emph{Proceedings of the {I}nternational {C}ongress of {M}athematicians,
  {V}ol. {I} ({B}eijing, 2002)} (Beijing), Higher Ed. Press, 2002, p.~539--546.

\bibitem{pink-roessler:2004}
\bysame , {\og On {$\psi$}-invariant subvarieties of semiabelian varieties and
  the {M}anin-{M}umford conjecture\fg}, \emph{J. Algebraic Geom.} \textbf{13}
  (2004), no.~4, p.~771--798.

\bibitem{poizat:2001}
{\scshape B.~Poizat} -- {\og L'\'egalit\'e au cube\fg}, \emph{J. Symbolic
  Logic} \textbf{66} (2001), no.~4, p.~1647--1676.

\bibitem{poonen99}
{\scshape B.~Poonen} -- {\og {M}ordell-{L}ang plus {B}ogomolov\fg},
  \emph{Invent. Math.} \textbf{137} (1999), no.~2, p.~413--245.

\bibitem{ran:1981}
{\scshape Z.~Ran} -- {\og On subvarieties of abelian varieties\fg},
  \emph{Invent. Math.} \textbf{62} (1981), no.~3, p.~459--479.

\bibitem{raynaud83b}
{\scshape M.~Raynaud} -- {\og Around the {M}ordell conjecture for function
  fields and a conjecture of {S}erge {L}ang\fg}, in \emph{Algebraic geometry
  (Tokyo/Kyoto, 1982)}, Lecture Notes in Math., no. 1016, Springer-Verlag,
  1983, p.~1--19.

\bibitem{raynaud83}
\bysame , {\og Courbes sur une vari{\'e}t{\'e} ab{\'e}lienne et points de
  torsion\fg}, \emph{Invent. Math.} \textbf{71} (1983), no.~1, p.~207--233.

\bibitem{remond2000c}
{\scshape G.~R{\'e}mond} -- {\og In\'egalit\'e de {V}ojta en dimension
  sup\'erieure\fg}, \emph{Ann. Scuola Norm. Sup. Pisa Cl. Sci. (4)} \textbf{29}
  (2000), no.~1, p.~101--151.

\bibitem{remond2001}
\bysame , {\og Sur le th\'eor\`eme du produit\fg}, \emph{J. Th\'eor. Nombres
  Bordeaux} \textbf{13} (2001), no.~1, p.~287--302, 21\textsuperscript e
  Journ\'ees Arithm\'etiques (Rome, 2001).

\bibitem{remond:2002}
\bysame , {\og Sur les sous-vari\'et\'es des tores\fg}, \emph{Compositio Math.}
  \textbf{134} (2002), no.~3, p.~337--366.

\bibitem{remond:2003}
\bysame , {\og Approximation diophantienne sur les vari\'et\'es
  semi-ab\'eliennes\fg}, \emph{Ann. Sci. \'Ecole Norm. Sup. (4)} \textbf{36}
  (2003), no.~2, p.~191--212.

\bibitem{remond:2005}
\bysame , {\og In\'egalit\'e de {V}ojta g\'en\'eralis\'ee\fg}, \emph{Bull. Soc.
  Math. France} \textbf{133} (2005), no.~4, p.~459--495.

\bibitem{remond:2005b}
\bysame , {\og Intersection de sous-groupes et de sous-vari\'et\'es. {I}\fg},
  \emph{Math. Ann.} \textbf{333} (2005), no.~3, p.~525--548.

\bibitem{remond:2007}
\bysame , {\og Intersection de sous-groupes et de sous-vari\'et\'es. {II}\fg},
  \emph{J. Inst. Math. Jussieu} \textbf{6} (2007), no.~2, p.~317--348.

\bibitem{remond:2009b}
\bysame , {\og Autour de la conjecture de {Z}ilber-{P}ink\fg}, \emph{J.
  Th\'eor. Nombres Bordeaux} \textbf{21} (2009), no.~2, p.~405--414.

\bibitem{remond:2009}
\bysame , {\og Intersection de sous-groupes et de sous-vari\'et\'es. {III}\fg},
  \emph{Comment. Math. Helv.} \textbf{84} (2009), no.~4, p.~835--863.

\bibitem{remond-viada:2003}
{\scshape G.~R{\'e}mond {\normalfont \smfandname} E.~Viada} -- {\og Probl\`eme
  de {M}ordell-{L}ang modulo certaines sous-vari\'et\'es ab\'eliennes\fg},
  \emph{Int. Math. Res. Not.} (2003), no.~35, p.~1915--1931.

\bibitem{roessler:2005}
{\scshape D.~Roessler} -- {\og A note on the {M}anin-{M}umford conjecture\fg},
  in \emph{Number fields and function fields---two parallel worlds}, Progr.
  Math., vol. 239, Birkh\"auser Boston, Boston, MA, 2005, p.~311--318.

\bibitem{sarnak-adams:1994}
{\scshape P.~Sarnak {\normalfont \smfandname} S.~Adams} -- {\og Betti numbers
  of congruence groups\fg}, \emph{Israel J. Math.} \textbf{88} (1994), no.~1-3,
  p.~31--72, With an appendix by Ze'ev Rudnick.

\bibitem{schlickewei:1997}
{\scshape H.~P. Schlickewei} -- {\og Lower bounds for heights on finitely
  generated groups\fg}, \emph{Monatsh. Math.} \textbf{123} (1997), no.~2,
  p.~171--178.

\bibitem{serre68b}
{\scshape J.-P. Serre} -- \emph{Abelian $\ell$-adic representations and
  elliptic curves}, W. A. Benjamin, Inc., New York-Amsterdam, 1968.

\bibitem{serre1997}
\bysame , \emph{Lectures on the {M}ordell-{W}eil theorem}, third \smfedname,
  Aspects of Mathematics, Friedr. Vieweg \& Sohn, Braunschweig, 1997,
  Translated from the French and edited by Martin Brown from notes by Michel
  Waldschmidt, With a foreword by Brown and Serre.

\bibitem{siu1993}
{\scshape Y.~T. Siu} -- {\og An effective {M}atsusaka big theorem\fg},
  \emph{Ann. Inst. Fourier (Grenoble)} \textbf{43} (1993), no.~5,
  p.~1387--1405.

\bibitem{ullmo98}
{\scshape E.~Ullmo} -- {\og Positivit{\'e} et discr{\'e}tion des points
  alg{\'e}briques des courbes\fg}, \emph{Ann. of Math.} \textbf{147} (1998),
  no.~1, p.~167--179.

\bibitem{viada:2008}
{\scshape E.~Viada} -- {\og The intersection of a curve with a union of
  translated codimension-two subgroups in a power of an elliptic curve\fg},
  \emph{Algebra Number Theory} \textbf{2} (2008), no.~3, p.~249--298.

\bibitem{viada:2009b}
\bysame , {\og The optimality of the {B}ounded {H}eight {C}onjecture\fg},
  \emph{J. Th\'eor. Nombres Bordeaux} \textbf{21} (2009), no.~3, p.~769--784.

\bibitem{viada:2010}
\bysame , {\og Lower bounds for the normalized height and non-dense subsets of
  subvarieties of abelian varieties\fg}, \emph{Int. J. Number Theory}
  \textbf{6} (2010), no.~3, p.~471--499.

\bibitem{vojta1996}
{\scshape P.~Vojta} -- {\og Integral points on subvarieties of semiabelian
  varieties. {I}\fg}, \emph{Invent. Math.} \textbf{126} (1996), no.~1,
  p.~133--181.

\bibitem{zannier:2009}
{\scshape U.~Zannier} -- \emph{Lecture notes on {D}iophantine analysis},
  Appunti. Scuola Normale Superiore di Pisa (Nuova Serie) [Lecture Notes.
  Scuola Normale Superiore di Pisa (New Series)], vol.~8, Edizioni della
  Normale, Pisa, 2009, With an appendix by Francesco Amoroso.

\bibitem{zhang95}
{\scshape S.-W. Zhang} -- {\og Positive line bundles on arithmetic
  varieties\fg}, \emph{J. Amer. Math. Soc.} \textbf{8} (1995), p.~187--221.

\bibitem{zhang98}
\bysame , {\og Equidistribution of small points on abelian varieties\fg},
  \emph{Ann. of Math.} \textbf{147} (1998), no.~1, p.~159--165.

\bibitem{zilber:2002}
{\scshape B.~Zilber} -- {\og Exponential sums equations and the {S}chanuel
  conjecture\fg}, \emph{J. London Math. Soc. (2)} \textbf{65} (2002), no.~1,
  p.~27--44.

\end{thebibliography}

\end{document}